\documentclass[reqno,final]{siamart}
\usepackage{color} 
\usepackage{amssymb} 
\usepackage{amsmath} 
\usepackage{epsfig}
\usepackage{listings}
\usepackage{amsmath,amsfonts} 
\usepackage{matlab-prettifier}
\usepackage{algorithm}
\usepackage{algorithmic}
\usepackage[T1]{fontenc} \usepackage{ae} \usepackage{aecompl}

\lstMakeShortInline[style=Matlab-editor]@
\usepackage{caption, subcaption}
\usepackage{graphicx}
\usepackage{cite}
\usepackage{tkz-graph}
\newtheorem{thm}{Theorem}[section]
\newtheorem{lem}{Lemma}[section]

\DeclareMathAlphabet\mathbfcal{OMS}{cmsy}{b}{n}

\newcommand{\Ah}{\widehat{\mathbf{A}}}

\def\A{\mathbf{A}}
\def\B{\mathbf{B}}

\def\D{\mathbf{D}}
\def\E{\mathbf{E}}
\def\F{\mathbf{F}}

\def\FF{{\mathbfcal{F}}}

\def\HH{\mathbfcal{H}_\text{m}}
\def\hatHH{\widehat{\mathbfcal{H}}_\text{m}}
\def\eye{\mathbf{I}}

\def\M{\mathbf{M}}

\def\P{\mathbf{P}}
\def\Q{\mathbf{Q}}
\def\R{\mathbf{R}}
\def\S{\mathbf{S}}
\def\T{\mathbf{T}}
\def\U{\mathbf{U}}
\def\V{\mathbf{V}}
\def\W{\mathbf{W}}
\def\Wn{{\mathbf{W}^\text{noise}}}
\def\WW{{\mathcal{W}}}

\def\Y{\mathbf{Y}}
\def\Z{\mathbf{Z}}
\def\randmatrix{\boldsymbol{\Omega}}
\def\noisevar{{\vec{\Gamma}}_\text{noise}^{-1}}
\def\eig{\mathbf{\Lambda}}
\def\post{\mathbf{\Gamma}_\text{post}}
\def\MMprior{\mathbf{Z}}
\def\reals{\mathbb{R}}

\def\domain{\mathcal{D}}

\newcommand{\ip}[2]{\langle {#1}, {#2}\rangle}
\newcommand{\mip}[2]{\langle {#1}, {#2}\rangle_{\scriptscriptstyle{\M}}}
\newcommand{\mnorm}[1]{ \| {#1} \|_{\scriptscriptstyle{\M}}}

\renewcommand{\vec}[1]{{\mathchoice
                     {\mbox{\boldmath$\displaystyle{#1}$}}
                     {\mbox{\boldmath$\textstyle{#1}$}}
                     {\mbox{\boldmath$\scriptstyle{#1}$}}
                     {\mbox{\boldmath$\scriptscriptstyle{#1}$}}}}

\DeclareMathAlphabet      {\mathup}{OT1}{\familydefault}{m}{n}

\newcommand{\Ns}{{n_s}}
\newcommand{\Nt}{{n_t}}
\newcommand{\ipar}{m}
\newcommand{\dpar}{\vec{m}}
\newcommand{\Cprior}{\Gamma_{\mathup{pr}}}
\newcommand{\Cpost}{\Gamma_{\mathup{post}}}
\newcommand{\ncov}{\vec{\Gamma}_{\!\mathup{noise}}}
\newcommand{\like}{\pi_{\mathup{like}}}
\newcommand{\obs}{\vec{y}}
\newcommand{\GM}[2]{\mathcal{N}({#1}, {#2})}
\newcommand{\priorm}{\mu_{\mathup{pr}}}
\newcommand{\postm}{\mu_{\mathup{post}}^{{\obs}}}
\newcommand{\iparpr}{\ipar_{\mathup{pr}}}
\newcommand{\iparmap}{\ipar_{\mathup{post}}^\obs}
\newcommand{\dparpr}{\dpar_{\mathup{pr}}}
\newcommand{\dparmap}{\dpar_{\mathup{post}}^\obs}

\newcommand{\mb}[1]{\mathbb{#1}}
\newcommand{\bmat}[1]{\begin{bmatrix} #1 \end{bmatrix}}
\def\scriptO{{{\it O}\kern -.42em {\it `}\kern + .20em}}
\newcommand{\prior}[1]{\mathbf{\Gamma}^{#1}_{\text{pr}}}
\newcommand{\expectation}[1]{\mathbb{E}\left[#1\right]}

\newcommand{\aopt}[1]{\Phi_\text{aopt}{#1}}
\newcommand{\aoptrand}[1]{\widehat{\Phi}_\text{aopt}{#1}}
\newcommand{\gradaoptrand}[1]{\widehat{\partial_j\Phi}_\text{aopt}{#1}}
\newcommand{\moda}[1]{\Phi_\text{mod}{#1}}
\newcommand{\modarand}[1]{\widehat{\Phi}_\text{mod}{#1}}

\newcommand{\trace}[1]{\mathsf{tr}\,\left(#1\right)}
\newcommand{\diag}{\mathsf{diag}\,}

\DeclareMathOperator*{\argmin}{arg\,min}

\title{Randomization and reweighted $\ell_1$-minimization for A-optimal
design of linear inverse problems\thanks{This work was funded, in part, by the National Science Foundation through the grant DMS-1745654.}}

\author{
Elizabeth Herman\thanks{Department of Mathematics, North Carolina State University,
Raleigh, NC 27695-8205, USA }
\and
Alen Alexanderian\footnotemark[2] 
\and	
Arvind K. Saibaba\footnotemark[2] 
}

\begin{document}

\maketitle

\begin{abstract} 
We consider optimal design of PDE-based Bayesian linear inverse problems with
infinite-dimensional parameters. We focus on the A-optimal design criterion,
defined as the average posterior variance and quantified by the trace of the
posterior covariance operator.   We propose using structure exploiting
randomized methods to compute the A-optimal objective function and its
gradient, and provide a detailed analysis of the error for the proposed
estimators.    To ensure sparse and binary design vectors, we develop a novel
reweighted $\ell_1$-minimization algorithm. We also introduce a modified
A-optimal criterion and present randomized estimators for its efficient
computation.  We present numerical results illustrating the proposed methods 
on a model contaminant source identification problem, where the
inverse problem seeks to recover the initial state of a contaminant plume
using discrete measurements of the contaminant in space and time.  
\end{abstract}

\begin{keywords}
Bayesian Inversion; A-Optimal experimental design; Large-scale ill-posed inverse problems;
Randomized matrix methods; Reweighted $\ell_1$ minimization; Uncertainty quantification
\end{keywords}

\begin{AMS}
35R30;  
62K05;  
68W20;  
35Q62;  
65C60;  
62F15.  
\end{AMS}

%
%
\section{Introduction}
A central problem in scientific computing involves estimating parameters that
describe mathematical models, such as initial conditions, boundary conditions,
or material parameters.  This is often addressed by using experimental
measurements and a mathematical model to compute estimates of unknown model
parameters.  In other words, one can estimate the parameters by solving an
inverse problem.  Experimental design involves specifying the experimental
setup for collecting measurement data, with the goal of accurately recovering
the parameters of interest.  As such, optimal experimental design (OED) is an
important aspect of effective and efficient parameter estimation.  Namely, in
applications where collecting experimental measurements is expensive (e.g.,
because of budget, labor, or physical constraints), deploying experimental
resources has to be done efficiently and in a parsimonious manner.  Even when
collecting large amounts of data is feasible, OED is still important; the
computational cost of processing all the data may be prohibitive or a poorly
designed experiment with many measurements may miss important information about
the parameters of interest.

To make matters concrete, we explain the inverse problem and the experimental
design in the context of an application. Consider the transport of a
contaminant in an urban environment or the subsurface. The forward problem
involves forecasting the spread of the contaminant, whereas the inverse problem
involves using the measurements of the contaminant concentration at discrete
points in space and time to recover the source of the contaminant (i.e., the
initial state). In this application, OED involves optimal placement of sensors,
at which measurement data is collected, to reconstruct the initial state. 

We focus on OED for Bayesian linear inverse problems governed by PDEs.  In our
formulation of the OED problem, the goal is to find an optimal subset of
sensors from a fixed array of $\Ns$ candidate sensor sites. The experimental
design is parameterized by assigning non-negative weights to each candidate
sensor location.  Ideally, we seek a binary weight vector $\vec{w}$; if $w_i =
1$, a sensor will be placed at the $i$th candidate location, and if $w_i = 0$,
no sensor will be placed at the location.  However, formulating an optimization
problem over binary weight vectors leads to a problem with combinatorial
complexity that is computationally prohibitive.  A common approach to address
this issue is to relax the binary requirement on design weights by letting the
weights take values in the interval $[0, 1]$. The sparsity of the design will
then be controlled using a penalty method; see
e.g.,~\cite{HaberMagnantLuceroEtAl12, a14infinite,yu2018scalable}. This results
in an optimization problem of the following form:
\begin{equation}\label{equ:basic_prob}
   \min_{\vec{w} \in [0, 1]^\Ns} \Phi(\vec{w}) + \gamma P(\vec{w}),
\end{equation}
where $\Phi$ denotes the design criterion, $\gamma > 0$
is a penalty parameter, and $P$ is a penalty function.

Adopting a Bayesian approach to the inverse problem, the design criterion will be a
measure of the uncertainty in the estimated parameters.  In this article, we
focus on a popular choice known as the A-optimal
criterion~\cite{Ucinski05,ChalonerVerdinelli95}. That is, we seek a sensor
configuration that results in a minimized average posterior variance.  The design
criterion, in this case, is given by the trace of the posterior covariance
operator. 

One major challenge in solving~\cref{equ:basic_prob}, specifically for PDE-based inverse
problems, is the computational cost of objective function and gradient
evaluations.  Namely, the posterior covariance operator $\post$ is dense,
high-dimensional, and computationally challenging to explicitly
form---computing applications of $\post$ to vectors requires solving multiple
PDEs.  Furthermore, these computations must be
performed at each iteration of an optimization algorithm used to
solve~\cref{equ:basic_prob}.  To address this computational challenge,
efficient and accurate matrix-free approaches for computing the OED objective
and its gradient are needed.  Another challenge in solving the OED
problem~\cref{equ:basic_prob} is the need for a suitable penalty method that is
computationally tractable and results in sparse and binary optimal weight
vectors.  This article is about methods for overcoming these computational
challenges.

\paragraph{Related work} For an extensive review of the OED literature, we
refer the reader to~\cite{alexanderian2018dopt}. We focus here on works that
are closely related to the present article.  Algorithms for A-optimal designs for
ill-posed linear inverse problems were proposed
in~\cite{HaberHoreshTenorio08,HaberMagnantLuceroEtAl12,FJ_16} and more specifically
for infinite-dimensional Bayesian linear inverse problems
in~\cite{a14infinite}. In majority of these articles, Monte-Carlo trace estimators are used
to approximate the A-optimal design criterion and its gradient.
Also,~\cite{HaberMagnantLuceroEtAl12,a14infinite} advocate use of low-rank
approximations using the Lanczos algorithm or the randomized
SVD~\cite{Halko2011structure}. We refer to our previous
work~\cite{saibaba2016randomized} for comparison of Monte Carlo trace
estimators and those based on randomized subspace iteration; it was shown that
the latter are significantly more accurate than Monte Carlo trace estimators.  
Regarding sparsity control, various techniques have been used to approximate
the $\ell_0$-``norm'' to enforce sparse and binary designs.  For example,
\cite{HaberHoreshTenorio08,HaberHoreshTenorio10,HaberMagnantLuceroEtAl12}  use
the $\ell_1$-penalty function with an appropriate threshold to post-process the solution. In \cite{a14infinite}, a continuation approach is proposed that
involves solving a sequence of optimization problems with non-convex penalty
functions that approximate the $\ell_0$-``norm''. More recently,
in~\cite{yu2018scalable}, a  sum-up rounding approach is proposed to obtain
binary optimal designs.

\paragraph{Our approach and contributions} 
In this article, we make the following advances in methods for A-optimal
sensor placements in infinite-dimensional Bayesian linear inverse problems: 

\begin{enumerate} 
\item We present efficient and accurate randomized
estimators of the A-optimal criterion and its gradient, based on randomized subspace iteration.  This is accompanied by a detailed algorithm that guides
efficient implementations, discussion of computational cost, as well as
theoretical error analysis; see~\cref{sec:criterion}.  Our estimators are
structure exploiting, in that they use the low-rank structure embedded in the
posterior covariance operator.   To quantify the accuracy of the estimators we present
rigorous error analysis, significantly advancing the methods
in~\cite{saibaba2016randomized}. A desirable feature of our analysis is that
the bounds are independent of the dimension of the discretized inversion
parameter.  Furthermore, the computational cost {(measured in the number
of PDE solves)} of the A-optimal objective and gradient  using our
proposed estimators is independent of the discretized parameter dimension.

\item We propose a new algorithm for optimal sensor placement that is based on
solving a sequence of reweighted $\ell_1$-optimization problems;
see~\cref{sec:reweightl1}.  An important benefit of this approach is that one
works with convex penalty functions, and since the A-optimal criterion itself
is a convex function of $\vec{w}$, in each step of the
reweighted $\ell_1$ algorithm a convex optimization problem is solved.  We
derive this algorithm by using the Majorization-Minimization principle applied
to a novel penalty function that promotes binary designs. The
solution of the
reweighted $\ell_1$-optimization problems is accelerated by the efficient
randomized estimators for the optimality criterion and its gradient.  To our
knowledge, the presented framework, based on reweighted $\ell_1$-minimization,
is the first of its kind in the context of OED.

\item Motivated by reducing computational cost, we propose a new criterion
known as modified A-optimal criterion; see~\cref{sec:mod_a}. This criterion is
derived by considering a suitably weighted A-optimal criterion.  We
present randomized estimators with complete error analysis for computing the
modified A-optimal criterion and its gradient.
\end{enumerate}
We illustrate the benefits of the proposed algorithms on a model problem from
contaminant source identification. A comprehensive set of numerical
experiments is provided so to
test various aspect of the presented approach; see~\cref{sec:numerics}.

Finally, we remark that the randomized estimators and the
reweighted $\ell_1$ approach for promoting sparse and binary weights are of
independent interest beyond the application to OED.

%
%
\section{Preliminaries}\label{sec:background}

In this section, we recall the background material needed 
in the remainder of the article.

\subsection{Bayesian linear inverse problems} 

We consider a linear inverse problem of estimating
$\ipar$, using the model 
\[
   F \ipar + \vec{\eta} = \obs.
\]
Here $F$ is a linear parameter-to-observable map (also called 
the forward operator), 
$\vec{\eta}$ represents the measurement noise, and $\obs$ is
a vector of measurement data.  
The inversion parameter $\ipar$ is an element of 
$\mathcal{V} = L^2(\mathcal{D})$, where $\mathcal{D}$
is a bounded domain. 

\subsubsection*{The setup of the inverse problem}
To fully specify the inverse problem, we need to describe the
prior law of $\ipar$ and our
choice of data likelihood. For the prior, we choose a Gaussian measure
$\priorm=\GM{\iparpr}{\Cprior}$. We assume 
the prior mean $\iparpr$ is a  sufficiently
regular element of $\mathcal{V}$ and that the covariance operator 
$\Cprior:\mathcal{V} \to \mathcal{V}$ is a
strictly positive self-adjoint trace-class operator.  Following 
the developments in~\cite{Stuart10,Bui-ThanhGhattasMartinEtAl13,DashtiStuart17}, 
we use    
$\Cprior = \mathcal{A}^{-2}$ with $\mathcal{A}$ taken to be a Laplacian-like
operator~\cite{Stuart10}. This ensures that $\Cprior$ is trace-class in two
and three space dimensions. 

We consider the case where $F$ represents 
a time-dependent PDE and we assume  
observations are taken at $\Ns$ sensor locations at $\Nt$ points in time.
Thus, the vector of experimental data $\obs$ is an element of $\reals^{\Ns\Nt}$.
An application of the parameter-to-observable map, 
$F:\mathcal{V} \to \reals^{\Ns\Nt}$,  
involves a PDE solve followed by an application of a spatiotemporal observation
operator.
We assume a Gaussian distribution on the experimental noise, $\vec{\eta} 
\sim \GM{\vec{0}}{\ncov}$.
Given this choice of the noise model---additive and Gaussian---the 
likelihood probability density function
(pdf) is
\[
\like(\obs \mid \ipar) \propto \exp\left\{ -\frac12
\big(F\ipar - \obs\big)^\top\ncov^{-1} \big(F\ipar - \obs\big)\right\}.
\]
Furthermore, the solution of the Bayesian inverse problem---the posterior
distribution law $\postm$---is given by the Gaussian measure
$\postm =
\GM{\iparmap}{\Cpost}$ with 
\begin{equation}\label{equ:mean-cov} 
\iparmap =  \Cpost(F^*\ncov^{-1}\obs + \Cprior^{-1}\iparpr), \qquad \Cpost =  (F^* \ncov^{-1} F + \Cprior^{-1})^{-1}.
\end{equation}
Note that here the posterior mean $\iparmap$ coincides with the 
maximum a posteriori probability 
(MAP) estimator.
We refer to~\cite{Stuart10} for further details.

\subsubsection*{Discretization}
We use a continuous Galerkin finite element discretization approach for the 
governing PDEs, 
as well as the inverse problem. 
Specifically, our discretization of the Bayesian inverse problem 
follows the developments in~\cite{Bui-ThanhGhattasMartinEtAl13}.
The discretized parameter space in the present case is
$\mathcal{V}_n = \reals^n$ equipped with the inner product $\mip{\cdot}{\cdot}$
and norm $\mnorm{\cdot} = \mip{\cdot}{\cdot}^{1/2}$, where 
$\M$ is the
finite element mass matrix. Note that $\mip{\cdot}{\cdot}$ is 
the discretized $L^2(\mathcal{D})$ inner product. The discretized 
parameter-to-observable map is a linear  transformation $\F:\mathcal{V}_n \to 
\reals^{\Ns\Nt}$ {with adjoint $\F^*$ discussed below}. The discretized 
prior measure $\GM{\dparpr}{\prior{}}$ is obtained
by discretizing the prior mean and covariance operator, and 
the discretized posterior measure is given by $\GM{\vec{m}_\text{post}}{\post}$, with
\[
\post = \left(\F^*\vec{\Gamma}^{-1}_\text{noise}\F + \prior{-1}\right)^{-1},
\quad 
\dparmap = \post\left(\F^*\vec{\Gamma}^{-1}_\text{noise}\obs+\prior{-1}\vec{m}_\text{pr}
\right).
\]
We point out that the operator $\F^*\vec{\Gamma}^{-1}_\text{noise}\F$ is the Hessian 
of the data-misfit cost functional whose minimizer is the MAP point,
and is thus referred to as the data-misfit Hessian; see e.g.,~\cite{a14infinite}.

It is important to keep track of the inner products in the domains and ranges
of the linear mappings, appearing in the above expressions, when computing the
respective adjoint operators.  For the readers' convenience, in
\cref{fig:adjoints}, we summarize the different spaces that are important, the
respective inner products, and the adjoints of the linear transformations
defined between these spaces.
\begin{figure}[ht]\centering
\includegraphics[width=.35\textwidth]{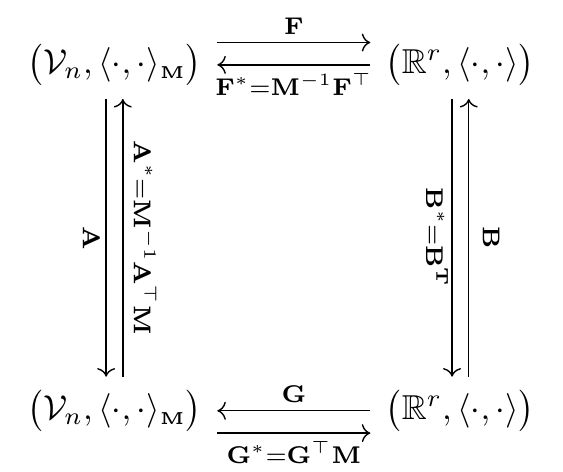}
\caption{Different spaces, their inner products, and the adjoints of linear transformations 
between them. Here $\ip{\cdot}{\cdot}$
denotes the Euclidean inner product and $\mip{\cdot}{\cdot}$ is the mass-weighted 
inner product.}
\label{fig:adjoints}
\end{figure}

Using the fact that $\prior{}$ is a self-adjoint operator on $\mathcal{V}_n$ 
and the form of the adjoint operator $\F^*$ (see \cref{fig:adjoints}), we can 
rewrite the expression for $\post$ as follows: 
\begin{equation*} 
\post = 
\prior{1/2}\M^{-1/2}\left(\eye+\FF^\top \noisevar\FF\right)^{-1}\M^{1/2}\prior{1/2}, 
\end{equation*}
with {
\begin{equation}\label{eqn:ff}
\FF=\F\prior{1/2}\M^{-1/2}.
\end{equation}
}
Note that {the operator $\HH = \FF^\top \noisevar\FF$ is a symmetric positive semidefinite matrix}, and is a similarity transform 
of the prior-preconditioned data-misfit Hessian 
$\prior{1/2}\F^*\vec{\Gamma}^{-1}_\text{noise}\F\prior{1/2}$. 
In many applications (including the application considered in \cref{sec:numerics}), $\HH$  
has rapidly decaying eigenvalues and therefore, it can be approximated by a
low-rank matrix. This is a key insight that will be exploited 
in our estimators for the OED criterion and its gradient.

\subsection{Randomized subspace iteration algorithm}\label{ssec:rand}
In this article, we develop and use randomized estimators to efficiently
compute the design criteria and their derivatives.  We first explain how to use
randomized algorithms for computing low-rank approximations of a 
symmetric positive semidefinite matrix
$\A\in\mb{R}^{n\times n}$. To draw
connection with the previous subsection, in our application  $\A$ will stand for $\HH$. We first draw a random Gaussian matrix $\boldsymbol\Omega \in
\mb{R}^{n\times \ell}$ (i.e., the entries are independent and identically
distributed standard normal random variables). We then perform $q$ steps of
subspace iteration on $\A$ with the starting guess $\boldsymbol\Omega$ to
obtain the matrix $\Y$. If, for example, the matrix
has rank $k \leq \ell$, or the eigenvalues decay sufficiently, then the range
of $\Y$ is a good approximation to the range of $\A$ under these suitable conditions. This is the main insight
behind randomized algorithms. 
We now show how to obtain a low-rank approximation of $\A$. A thin-QR factorization of $\Y$ is performed to obtain the matrix $\Q$, which has orthonormal columns. We then form the ``projected'' matrix $\T = \Q^\top\A\Q$ and obtain the low-rank approximation 
\begin{equation}
 \A \approx \Q \T\Q^\top.
 \label{approx}
 \end{equation}
This low-rank approximation can be manipulated in many ways depending on the desired application. An alternative low-rank approximation can be computed using the Nystr\"om approximation, see e.g.,~\cite{Halko2011structure}. 
\begin{algorithm}[h!]
\caption{Randomized subspace iteration.} 
\begin{algorithmic}
\REQUIRE $\A\in\mathbb{R}^{n\times n}$ with target rank $k$, 
oversampling parameter $p\geq 2$, with $\ell\equiv k+p \leq n$, and $q\geq 1$ 
(number of  subspace iterations). 
\ENSURE  $\Q \in \mb{R}^{n\times \ell}, \T\in\mathbb{R}^{\ell\times\ell}$.
\STATE \textbf{Draw}  a standard Gaussian random matrix $\randmatrix\in\reals^{n\times \ell}$. 
\STATE \textbf{Compute} $\Y=\A^q\boldsymbol\Omega$. 
\STATE \textbf{Compute} thin QR decomposition $\Y=\Q\R.$ 
\STATE \textbf{Compute} $\T=\Q^\top\A\Q$. 
\end{algorithmic}
\label{alg:randsubspace}
\end{algorithm}
In addition, once the matrix $\T$ is computed, it can be used in various ways. In
\cite{saibaba2016randomized}, $\trace{\T}$ was used as an estimator for
$\trace{\A}$, whereas $\log\det(\eye + \T)$ was used as estimator for
$\log\det(\eye + \A)$. The main idea behind these estimators is that the
eigenvalues of $\T$ are good approximations to the eigenvalues of $\A$, when
$\A$ is sufficiently low-rank or has rapidly decaying eigenvalues. Our
estimators for the A-optimal criterion and its gradient utilize the same idea
but in a slightly different form.

\subsection{A-optimal design of experiments}\label{ssec:oed}
As mentioned in the introduction, an experimental design refers to a placement 
of sensors used to collect measurement data for the purposes of parameter 
inversion. Here we describe the basic setup of the optimization problem
for finding an A-optimal design.

\subsubsection*{Experimental design and A-optimal criterion}
We seek to find an optimal subset of 
a network of $\Ns$ candidate sensor locations, which collect measurements
at $\Nt$ points in time. The experimental design is parameterized by a vector
of design weights $\vec{w} \in [0, 1]^\Ns$. 
In the present work, we use the A-optimal criterion to find the optimal design.  That is, we seek designs
that minimize the average posterior variance, as quantified by
$\trace{\post(\vec{w})}$.  
(The precise nature of the
dependence of $\post$ on $\vec{w}$ will be explained below.)
Note that 
$\trace{\post(\vec{w})} = \trace{ \post(\vec{w}) - \prior{}} + \trace{\prior{}}$
and thus, minimizing the trace of the posterior covariance operator is equivalent to minimizing
\begin{equation}\label{estobj}
    \aopt(\vec{w}) \equiv \trace{ \post(\vec{w}) - \prior{}}.
\end{equation}
This is the objective function we seek to minimize for finding A-optimal
designs. As seen below, this formulation of the A-optimal criterion 
is well suited for approximations via randomized matrix methods \cref{ssec:rand}.
Note also that minimizing $\aopt(\vec{w})$ amounts to maximizing
$\trace{\prior{}} - \trace{\post(\vec{w})}$, which can be thought of as a measure
of uncertainty reduction. 

We can also understand~\cref{estobj} from a decision theoretic 
point of view. It is well known~\cite{ChalonerVerdinelli95,
alexanderian2016bayesian,AttiaAlexanderianSaibaba18} that 
for Bayesian linear inverse problems with Gaussian prior and 
additive Gaussian noise models, $\trace{\post{}}$ 
coincides with Bayes risk (with respect to the $L^2$ loss function):
\[\begin{aligned}
    \trace{\post{}} = &\> \mathbb{E}_{\priorm} \Big(
    \mathbb{E}_{\like(\obs \mid \dpar)} \big(\mnorm{\dparmap - \dpar}^2 \big)
    \Big)\\ 
= &\> \int_{\mathcal{V}_n} \int_{\reals^{n_s n_t}} 
    \mnorm{\dparmap - \dpar}^2 \, \like(\obs \mid \dpar)d\obs\,\priorm(d\dpar) .
\end{aligned}\]
Here $\priorm$ denotes the discretized prior measure, 
$\priorm = \mathcal{N}(\vec{\dparpr}, \prior{})$.
Using
$$\int_{\mathcal{V}_n} \mnorm{\dparpr - \dpar}^2 \, \priorm(d\vec{m})
= \trace{\prior{}},$$
we see that 
\[
\trace{\post{} - \prior{}} = \mathbb{E}_{\priorm} \Big(\mathbb{E}_{\like(\obs \mid \dpar)} 
\big(\mnorm{\dparmap - \dpar}^2 -  
\mnorm{\dparpr - \dpar}^2 \big)
\Big);
\]
this provides an alternate interpretation of $\aopt(\cdot)$ as a Bayes 
risk with respect to a modified loss function.

\subsubsection*{Dependence of the A-optimal criterion to the design weights}
We follow the same setup as~\cite{a14infinite}.  Namely, the design weights enter
the Bayesian inverse problem through the data likelihood, resulting in the $\HH$ operator being dependent on $\vec{w}$.  To weight the spatiotemporal observations, the matrix $\W$ is defined as: 
\[
\W=\sum_{i=1}^{n_s}w_i\E_i, \quad \text{with} \quad
\E_i= \eye_{n_t}\otimes \mathbf{e}_i\mathbf{e}_i^\top,
\]
where $\otimes$  is the Kronecker product.  Therefore, $\HH(\vec{w})$ is expressed as follows:
\[
\HH(\vec{w}) = 
\FF^\top  \W^{1/2}\noisevar\W^{1/2}\FF,
\] 
We refer to~\cite{a14infinite} for details.
This results in the $\vec{w}$ dependent posterior covariance operator:
\[ 
\post(\vec{w}) 
= \prior{1/2}\M^{-1/2}\left(\HH(\vec{w}) + \eye\right)^{-1}\M^{1/2}
\prior{1/2}.
\]
In our formulation, 
we assume uncorrelated observations across sensor locations and time which implies $\vec{\Gamma}_\text{noise}$ is a $n_s n_t \times n_s n_t$ block diagonal matrix, 
with $n_t$ blocks of the form $\diag(\sigma_1^2, 
\ldots, \sigma_{n_s}^2)$; here $\sigma_i^2$, $i = 1, \ldots, n_s$, 
denote the measurement noise at individual sensor locations. 
Using this structure for $\vec{\Gamma}_\text{noise}$, we define 
\begin{equation}
\label{Wnoise}
\Wn \equiv \W^{1/2}\noisevar\W^{1/2} =
\sum_{i=1}^{n_\text{sens}}w_i\E_i^\text{noise},
\end{equation}
with $\E_i^\text{noise} = \sigma_i^{-2}\eye_{n_\text{time}}\otimes \mathbf{e}_i\mathbf{e}_i^\top$.
Thus, we have $\HH(\vec{w}) = \FF^\top\Wn\FF$ and the A-optimal criterion can be written as 
\begin{equation}\label{equ:criterion}
\begin{aligned}
\aopt{(\vec{w})} &= 
\trace{\prior{1/2}\M^{-1/2}\big[\big(\HH(\vec{w}) + \eye\big)^{-1} - \eye\big] \M^{1/2}\prior{1/2} }
\\
&= \trace{ \big[\big(\HH(\vec{w}) + \eye\big)^{-1} - \eye\big] \MMprior}, 
\end{aligned}
\end{equation}
with {
\begin{equation} \label{eqn:Z}
\MMprior\equiv\M^{1/2}\prior{}\M^{-1/2}.
\end{equation}
}

Anticipating that we will use a gradient-based solver for solving \cref{opt},
we also need the gradient of $\aopt{(\vec{w})}$ which we now derive. Using
Theorems B.17 and B.19 in~\cite{Ucinski05}, the partial derivatives of
\cref{equ:criterion} with respect to $w_j$, $j = 1,\dots, n_s$, are
\begin{equation}
\label{equ:gradient}
\begin{aligned}
\partial_j \aopt(\vec{w}) 
=-\trace{(\eye+\HH(\vec{w}))^{-1}\partial_j \HH(\vec{w})(\eye+\HH(\vec{w}))^{-1}\MMprior}.
\end{aligned}
\end{equation}
(We have used the notation $\partial_j$ to denote $\frac{\partial}{\partial w_j}$.)
Note that using the definition of $\HH(\vec{w})$, we have $\partial_j \HH(\vec{w})=\FF^\top\E_j^\text{noise}\FF$, $j=1,\ldots,n_s$.

\subsubsection*{The optimization problem for finding an A-optimal design}
We now specialize the optimization problem \cref{equ:basic_prob}
to the case of A-optimal sensor placement for linear inverse problems
governed by time-dependent PDEs:
\begin{equation}
\min_{\vec{w}\in[0, 1]^{\Ns}}\aopt(\vec{w})+\gamma P(\vec{w}).
\label{opt}
\end{equation}
As explained before, to enable efficient solution methods for the above
optimization problem we need (i) a numerical method for fast computation of
$\aopt(\vec{w})$ and its gradient, and (ii) a choice of penalty 
function that promotes sparse and binary weights. 
The former is facilitated by the randomized subspace iteration 
approach outlined earlier (see \cref{sec:criterion}), and 
for the latter we present an approach based
on rewighted $\ell_1$ minimization (see \cref{sec:reweightl1}).

%
%
\section{Efficient computation of A-optimal criterion and its gradient}\label{sec:criterion}

The  computational cost of solving \cref{opt} is dominated by the PDE solves
required in OED objective and gradient evaluations; these operations need to be
performed repeatedly when using an optimization algorithm for solving
\cref{opt}.  Therefore, to enable computing A-optimal 
designs for large-scale applications, efficient methods
for objective and gradient computations are needed. 
In this section, we derive efficient and accurate randomized
estimators for \cref{equ:criterion} and \cref{equ:gradient}. The proposed estimators are
matrix-free---they require only applications of the (prior-preconditioned)
forward operator $\FF$ and its adjoint on vectors. Moreover, the computational
cost of computing these estimators does not increase with the discretized
parameter dimension. This is due to the fact that our estimators exploit the
low-rank structure of $\HH(\vec{w})$, a problem property that is independent of the
choice of discretization.

We introduce our proposed randomized estimators for the A-optimal design
criterion $\aopt{(\vec{w})}$ and its gradient $\nabla\aopt{(\vec{w})}$ in
\cref{sec:aopt_estimators}. Additionally, we present a detailed computational
method for computing the proposed estimators.  We analyze the errors associated
with our proposed estimators in \cref{ssec:error}.

\subsection{Randomized estimators for $\aopt(\vec{w})$ and its gradient}
\label{sec:aopt_estimators}

Consider the low-rank approximation of $\HH(\vec{w})$ given by
$\hatHH(\vec{w})=\Q(\vec{w})\T(\vec{w})\Q^\top(\vec{w})$, with 
$\Q(\vec{w})$ and $\T(\vec{w})$ computed using \cref{alg:randsubspace}.
Replacing
$\HH(\vec{w})$ by its approximation and using the cyclic property
of the trace, we obtain the estimator for the A-optimal criterion \cref{equ:criterion}:
\begin{equation}\label{aoptest}
\aoptrand{(\vec{w};\ell)}=\trace{\left((\eye+\hatHH(\vec{w}))^{-1}-\eye\right) \MMprior},
\end{equation}
where $\MMprior$ is as in~\cref{eqn:Z}.

To derive an estimator for the gradient, once again, we 
replace $\HH(\vec{w})$ with its low-rank approximation $\hatHH(\vec{w})$ in 
\cref{equ:gradient} to obtain 
\begin{equation}
\gradaoptrand{(\vec{w};\ell)} = -\trace{(\eye+\hatHH(\vec{w}))^{-1}\FF^\top\E_j^\text{noise}\FF(\eye+\hatHH(\vec{w}))^{-1}\MMprior}, 
\label{estgrad}
\end{equation}
for $j=1,\dots n_s$.

\subsubsection*{Computational procedure}
First, we discuss computation of the A-optimal criterion estimator using
\cref{alg:randsubspace}.  Typically, the algorithm can be used
with $q=1$, due to rapid decay of eigenvalues of $\HH(\vec{w})$. In this 
case, \cref{alg:randsubspace} requires $2\ell$ applications of $\HH(\vec{w})$.
Since each application of
$\HH(\vec{w})$ requires one $\FF$ apply (forward solve) and one $\FF^\top$ apply
(adjoint solve), computing $\hatHH(\vec{w})$ requires $4\ell$ PDE solves.  Letting the
spectral decomposition of the $\ell \times \ell$ matrix $\T(\vec{w})$ be given by
$\T(\vec{w})=\U(\vec{w})\eig_\T(\vec{w}) \U(\vec{w})^\top$ and denoting $\V(\vec{w})\equiv\Q(\vec{w})\U(\vec{w})$, we have  $\hatHH(\vec{w})=\V(\vec{w})\eig_\T(\vec{w})
\V^\top(\vec{w})$.  Applying the Sherman--Morrison--Woodbury formula
\cite{meyer2000matrix} and the cyclic property of the trace to \cref{aoptest},
we obtain
\begin{equation}
\aoptrand{(\vec{w})}= -\trace{\D_\T(\vec{w}) \V^\top(\vec{w})\MMprior \V(\vec{w})},
\label{expandedobj}
\end{equation}
where $\D_\T(\vec{w})=\eig_\T(\vec{w})(\eye+\eig_\T(\vec{w}))^{-1}$.   
To simplify notation, the dependence of $\vec{w}$ is suppressed for the operators used in computing the estimators for the remainder of the article; however, the notation is retained for $\HH(\vec{w})$ and $\hatHH(\vec{w})$.

Next, we describe computation of the gradient estimator \cref{estgrad}. Here 
we assume $n_sn_t \leq n$; the extension to the case $ n_sn_t > n$ is straightforward and is omitted. Again, using the Woodbury formula and cyclic property of the trace, we rewrite \cref{estgrad} as 
\begin{equation}
\gradaoptrand{(\vec{w};\ell)} = -\trace{\FF(\eye - \V\D_\T\V^\top)\MMprior(\eye - \V\D_\T\V^\top)\FF^\top\E_j^\text{noise}}
\label{gradexpand}
\end{equation}
for $j=1,\dots,n_s$. Expanding this expression, we obtain 
\begin{equation}
\label{alt}
\begin{aligned}
\widehat{\partial_j\Phi}_\text{aopt}(\vec{w}) 
&= -\trace{\MMprior\FF^\top\E_j^\text{noise}\FF} + 2\trace{\FF\V\D_\T\V^\top\MMprior\FF^\top\E_j^\text{noise}} -\\
&\quad\trace{\FF\V\D_\T\V^\top\MMprior\V\D_\T\V^\top\FF^\top\E_j^\text{noise}}.
\end{aligned}
\end{equation}
Note that the first term $s_j=-\trace{\MMprior\FF^\top\E_j^\text{noise}\FF} $ in \cref{alt} does not depend on the design $\vec{w}$, for $j=1,\dots,n_s$.  As a result, this term can be precomputed and used in subsequent function evaluations. We expand $s_j$ to be
\[\begin{aligned}
s_j &=-\trace{\MMprior\FF^\top\E_j^\text{noise}\FF}\\
&=-\sum_{k=1}^{n_sn_t}\left(\FF^\top(\E_j^\text{noise})^{1/2}\widehat{\vec{e}}_k\right)^\top\MMprior\left(\FF^\top(\E_j^\text{noise})^{1/2}\widehat{\vec{e}}_k\right),
\end{aligned}\]
where $\widehat{\vec{e}}_k$ is the $k^\text{th}$ column of the identity matrix of size $n_sn_t$.  Because there are only $n_t$ columns of $\E_j^\text{noise}$ with nonzero entries, the total cost to precompute $s_j$ for $j=1,\dots,n_s$ is $n_sn_t$ PDE solves.  
 
To compute the remaining terms in \cref{alt}, we exploit the fact that
$\V$ has $\ell$ columns.  Notice all the other occurrences of $\FF$ and
$\FF^\top$ occur as a combination of $\FF\V$ and $\FF\MMprior\V$ (or of their
transposes).  Both of these terms require $\ell$ PDE solves to compute.
As a result, the total cost to evaluate $\aoptrand{(\vec{w})}$ and
$\widehat{\nabla\Phi}_\text{aopt}(\vec{w};\ell)$ is
$4\ell$ PDE solves to apply \cref{alg:randsubspace} and $2\ell$ PDE solves to
compute $\FF\V$ and $\FF\MMprior\V$. We detail the steps for computing our 
estimators for A-optimal criterion and its gradient in~\cref{alg:randobjgrad}.

\begin{algorithm}[ht!]
\caption{Randomized method for computing $\aoptrand{(\vec{w};\ell)}$ and $\widehat{\nabla \Phi}_\text{aopt}(\vec{w};\ell)$.} 
\begin{algorithmic}[1]
\REQUIRE Target rank $k$, oversampling parameter $p \geq 0$, design $\vec{w}$, and $s_j$ for $j=1,\dots,n_s$.
\ENSURE  OED objective $\aoptrand{(\vec{w};\ell)}$ and gradient $\widehat{\nabla\Phi}_\text{aopt}(\vec{w};\ell)$.
\STATE Apply \cref{alg:randsubspace} with $\ell = k + p$ and $q=1$ to obtain $\T \in \reals^{\ell \times \ell}$ and $\Q \in \reals^{n\times \ell}$. 
\STATE Compute eigendecomposition $[\U,\eig_\T]$ of $\T$. Let $\D_\T=\eig_\T(\eye+\eig_\T)^{-1}$.
\FOR{$i=1$ to $\ell$}
   \STATE Compute $\vec{v_i} \!=\! \Q \vec{u_i}$, where $\vec{u_i}$ are the columns of $\U$. 
\ENDFOR
\STATE Compute
$$\aoptrand{(\vec{w};\ell)} = -\sum_{i=1}^\ell d_i\vec{v}_i^\top \MMprior\vec{v}_i,$$
 where $d_i$ is the $i^\text{th}$ diagonal of $\D_\T$.
 \FOR{$i=1$ to $\ell$}
 \STATE Compute $\vec{a}_i=\FF\vec{v}_i$ and $\vec{b}_i=\FF\MMprior\vec{v}_i$.
 \ENDFOR
\FOR{$j=1$ to $n_s$}
\STATE Compute 
\[
\widehat{\partial_j \Phi}_\text{aopt}(\vec{w};\ell) = s_j+2\sum_{i=1}^\ell d_i\vec{b}_i^\top\E_j^\text{noise}\vec{a}_i - \sum_{i=1}^\ell \sum_{k=1}^\ell d_id_k\vec{a}_i^\top\E_j^\text{noise}\vec{a}_k\vec{v}_k^\top\MMprior\vec{v}_i.
\]
\ENDFOR
\STATE Return $\aoptrand{(\vec{w};\ell)}$ and $\widehat{\nabla\Phi}_\text{aopt}(\vec{w};\ell)=[\widehat{\partial_1 \Phi}_\text{aopt}(\vec{w};\ell),\dots,\widehat{\partial_{n_s} \Phi}_\text{aopt}(\vec{w};\ell)]^\top$.
\end{algorithmic}
\label{alg:randobjgrad}
\end{algorithm}

\subsubsection*{Alternative approaches and summary of computational cost}
A closely related variation of \cref{alg:randobjgrad} is obtained by replacing
step 1 of the algorithm (i.e., randomized subspace iteration) by the
solution of an eigenvalue problem to compute the dominant eigenvalues of
$\HH(\vec{w})$ ``exactly''. We refer to this method as Eig-$k$, where
$k$ is the target rank of $\HH(\vec{w})$.  
This idea was explored for computing Bayesian D-optimal designs in
\cite{alexanderian2018dopt}.  The resulting cost is similar to that of the
randomized method: it would cost $\mathcal{O}(k)$ PDE solves per iteration 
to compute the spectral decomposition of $\HH(\vec{w})$, plus  
$\min\{n_sn_t,n\}$ PDE solves to precompute $s_j$ for $j=1,\dots,n_s$.  While both the randomized and Eig-$k$ methods provide a viable scheme for computing the A-optimal criteria, our randomized method can exploit parallelism to lower computational costs.  Each matrix-vector application  with $\HH(\vec{w})$ in \cref{alg:randsubspace} can be computed in parallel.  However, if accurate eigenpairs of $\HH(\vec{w})$ are of importance to the problem, one can choose to use the Eig-$k$ approach at the cost of computing a more challenging problem.

Another possibility suitable for problems where the forward model does not
depend on the design (as is the case in the present work), is to precompute a
low-rank SVD of $\FF$, which can then be applied as necessary to compute the
A-optimal criterion and its gradient.  This \emph{frozen forward operator
approach} has been explored in \cite{a14infinite,HaberMagnantLuceroEtAl12} for
the A-optimal criterion and in \cite{alexanderian2018dopt} for the D-optimal
criterion.  The resulting PDE cost of precomputing a low-rank approximation of
$\FF$ is $\mathcal{O}(k)$, with $k$ indicating the target rank.  The Frozen
method is beneficial as no additional PDE solves are required when applying
\cref{alg:randobjgrad}; however, this approach would not favor problems
where $\FF$ depends on $\vec{w}$, nor can the modeling errors associated with
the PDE be controlled in subsequent evaluations without the construction of
another operator.

Finally, if the problem size is not too large (i.e., in small-scale
applications), one could explicitly construct the forward operator $\FF$. This enables exact
(excluding floating point errors) computation of the objective and its gradient.
The total cost for evaluating $\FF$ involves an upfront cost of $\min\{n_sn_t,n\}$ PDE solves.
We summarize the computational cost of \cref{alg:randobjgrad} along with 
the other alternatives mentioned above in \cref{table:cost}.  
\begin{table}[h!]
\begin{center}
\begin{tabular}{c|c|c|c}
Method & $\Phi_\text{aopt}(\vec{w})$ and $\nabla\Phi_\text{aopt}(\vec{w})$ &  Precompute & Storage Cost\\
\hline
``Exact'' & $-$ & $\min\{n_sn_t,n\}$ & $nn_sn_t$\\
Frozen & $-$ & $\mathcal{O}(k)$ & $(2n+1)k$\\
Eig-$k$ & $\mathcal{O}(k)$ & $\min\{n_sn_t,n\}$ & $n_s$\\
Randomized  & $2(q+2)(k+p)$ & $\min\{n_sn_t,n\}$ & $n_s$
\end{tabular}
\caption{Computational cost measured in terms of PDE solves for different methods in computing  $\Phi_\text{aopt}(\vec{w})$ and $\nabla\Phi_\text{aopt}(\vec{w})$.  Typically, $q=1$ in \cref{alg:randsubspace} is sufficient.}
\label{table:cost}
\end{center}
\end{table}

To summarize, the randomized methods for computing OED objective and its
gradient present several attractive features; our approach is well suited to
large-scale applications; it is matrix free, simple to implement and
parallelize, and exploits low-rank structure in the inverse problem. Moreover,
as is the case with the Eig-k approach, the performance of the randomized
subspace iteration does not degrade as the dimension of the parameter increases
due to mesh refinement.  This is the case because the randomized subspace
iteration relies on spectral properties of the prior-preconditioned data misfit
Hessian---a problem structure that is independent of discretization.

\subsection{Error Analysis}\label{ssec:error}
Here we analyze the error associated with our estimators computed 
using \cref{alg:randobjgrad}.  
For fixed $\vec{w}\in[0,1]^{n_s}$, since $\HH(\vec{w}) \in \mb{R}^{n\times n}$ is symmetric
positive semidefinite, we can order its eigenvalues as 
$\lambda_1\geq\lambda_2\geq\dots\geq\lambda_n\geq 0$. Suppose $\lambda_1,
\ldots, \lambda_k$ are the dominant eigenvalues of $\HH(\vec{w})$, we define $\eig_1 =
\diag(\lambda_1,\dots,\lambda_k)$ and $\eig_2=
\diag(\lambda_{k+1},\dots,\lambda_n)$ and we assume that the eigenvalue ratio satisfies 
$$\gamma_k = \|\eig_2\|_2\|\eig_1^{-1}\|_2 = \frac{\lambda_{k+1}}{\lambda_k} < 1.$$

We now present the error bounds for the objective function and its gradient. To this end, we define the constant $C$ as 
\begin{equation}
 C \equiv \frac{e^2(k+p)}{(p+1)^2}\left( \frac{1}{2\pi(p+1)}\right)^{\frac{2}{p+1}} \left( \mu + \sqrt{2}\right)^2 \left( \frac{p+1}{p-1}\right),
\label{constant}
\end{equation}
with  $r = \text{rank}(\HH(\vec{w}))$ and  $\mu \equiv \sqrt{r-k} + \sqrt{k+p}$.

\begin{theorem}
\label{objerror}
Let $\aoptrand{(\vec{w};\ell)}$ and $\widehat{\nabla\Phi}_\text{aopt}(\vec{w};\ell)$ be approximations of the A-optimal objective function $\aopt{(\vec{w})}$ and its gradient $\nabla\aopt{(\vec{w})}$, respectively, computed using \cref{alg:randobjgrad} for fixed $\vec{w}\in[0,1]^{n_s}$.  Recall that $k$ is the target rank of $\HH(\vec{w})$, $p \geq 2$ is the oversampling parameter such that $k+p \leq n$, and $q \geq 1$ is the number of subspace iterations.  Assume that $\gamma_k < 1$.  Then, with $f = x/(1+x)$ 
\begin{equation}\label{equ:aopt_obj_bound}
\expectation{|\aopt{(\vec{w})}-\aoptrand{(\vec{w};\ell)}|}\leq \|\MMprior\|_2\left(\trace{f(\eig_2)}  +  \trace{f(\gamma_k^{2q-1}C\eig_2)}\right).
\end{equation} 
Furthermore, with $\P_j = \FF^\top\E_j^\text{noise} \FF$, for $j=1,\dots, n_s$,
\begin{equation}\label{equ:aopt_grad_bound}
\expectation{|\partial_j \aopt{(\vec{w})}-\partial_j \aoptrand{(\vec{w};\ell)}|}\leq  \>  2\|\MMprior\|_2\|\P_j \|_2\left(\trace{f(\eig_2)} + \trace{f(\gamma_k^{2q-1}C\eig_2)}\right).
\end{equation}
 \end{theorem}
\begin{proof} See \cref{proofobjerror}.
\end{proof}

In \cref{objerror}, the estimators are unbiased when the target
rank equals the rank of $\HH(\vec{w})$.  If the eigenvalues decay rapidly, the bounds suggest
that the estimators are accurate. Recall that $\text{rank}(\HH(\vec{w}))\leq 
\min\{n_sn_t,n\}$ is the number of nonzero eigenvalues. Consequently, it is
seen that the bounds are independent of the dimension of the discretization
$n$.

%
%
\section{An optimization framework for finding binary designs}\label{sec:reweightl1}

We seek A-optimal designs by solving an optimization problem of the form, 
\begin{equation}
\label{opt_relaxed}
\min_{\vec{w}\in \mathcal{W}}
\Phi(\vec{w})+\gamma P(\vec{w}), \quad \mathcal{W} = [0, 1]^{n_s},
\end{equation}
where the design criterion $\Phi$ is either the A-optimal criterion $\aopt(\vec{w})$ or the modified A-optimal criterion $\moda(\vec{w})$ (see \cref{sec:mod_a}). In the previous sections, we laid out an efficient
framework for computing accurate approximations to the A-optimal criterion and its gradient. We now discuss the choice of the penalty
term $P$ and the  algorithm for solving the optimization problem. 

The choice of the penalty term must satisfy two conditions: enforcing 
sparsity, measured
by the number of nonzeros of the design vector, and binary designs, i.e.,
designs vectors whose entries are either $1$ or $0$. 
One possibility for the penalty function is
the $\ell_0$-``norm'', $P_{\ell_0}(\vec{w}) = \| \vec{w} \|_0$, which measures
the number of nonzero entries in the design.  However, the resulting
optimization problem is challenging to solve due to its combinatorial
complexity. A common practice is to replace the $\ell_0$-``norm'' penalty by
the $\ell_1$-norm, $P_{\ell_1}(\vec{w}) = \| \vec{w} \|_1$.  The penalty
function  $P_{\ell_1}$ has desirable features: it is a convex penalty function
that promotes sparsity of the optimal design vector $\vec{w}$.
However, the resulting design is
sparse but not necessarily binary and  additional post-processing in the form
of thresholding is necessary to enforce binary designs.

In what follows, we introduce a suitable penalty function that enforces both
sparsity and binary designs and an algorithm for solving the OED optimization
problem based on the MM approach.  The resulting algorithm takes the form of a
reweighted $\ell_1$-minimization algorithm~\cite{Candes2008sparsity}.

\subsection{Penalty functions} We propose the following penalty function \begin{equation}
P_{\epsilon}(\vec{w})=\sum_{i=1}^{n_s}\displaystyle{\frac{|w_i|}{|w_i|+\epsilon}},
\quad \vec{w} \in \reals^{n_s},
\label{pep}
\end{equation}
for a user-defined parameter $\epsilon>0$. This penalty function approximates $P_{\ell_0}$ for small values of
$\epsilon$; however, as $\epsilon$ becomes smaller the corresponding optimization problem becomes harder. 
To illustrate the choice of penalty functions, in~\cref{fig:norm}, we plot $P_{0.05}$ 
along with
$P_{\ell_0}$ and $P_{\ell_1}$, with $n_s = 1$. Using $P_\epsilon$ in the OED problem leads to the optimization problem, 
\begin{equation}
\min_{\vec{w}\in\mathcal{W}}\Phi(\vec{w})+\gamma P_\epsilon(\vec{w}).
\label{eqn:pepsilon}
\end{equation}
In \cref{eqn:pepsilon}, the absolute values in definition of 
$P_\epsilon$ can be dropped since we limit
the search for optimal solutions in $\WW$.  Since $P_{\epsilon}(\vec{w})$ is
concave, \cref{eqn:pepsilon} is a non-convex optimization problem.
To tackle this, we adopt the majorization-minimization (MM) approach.

\begin{figure}[!ht]
\begin{center}
  \includegraphics[width=0.5\textwidth]{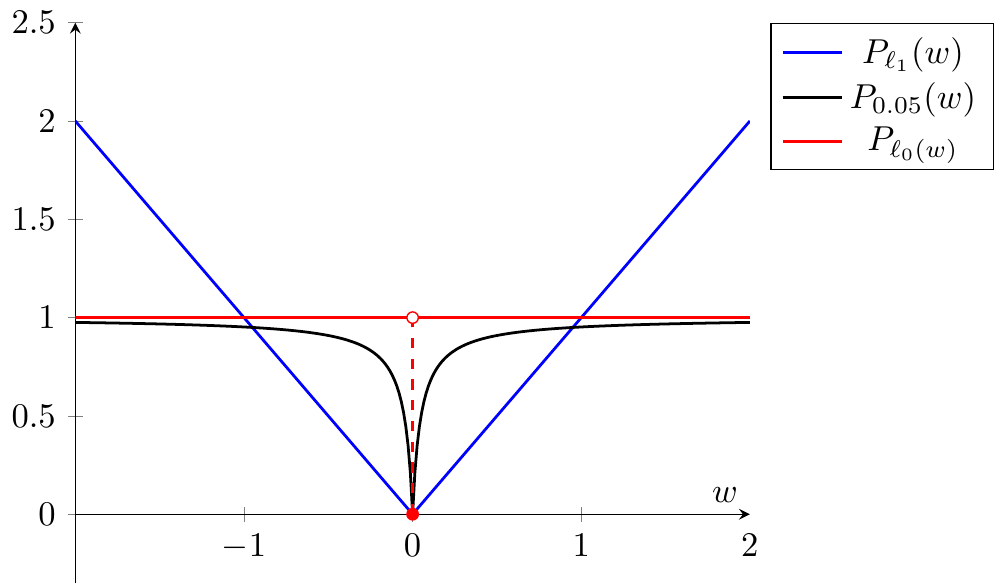}
\end{center}
\caption{Different choices of penalty functions with $\Ns = 1$.}
\label{fig:norm}
\end{figure}

\subsection{MM approach and reweighted $\ell_1$ algorithm} The idea behind the
MM approach is to solve a sequence of optimization problems whose solutions
converge to that of the original problem~\cite{hl2004MMAlgorithms,lange2016mm}.
This sequence is generated by a carefully constructed surrogate that satisfies
two properties---the surrogate must majorize the objective function for all
values, and the surrogate must match the objective function  at the current
iterate. More specifically, suppose
\[ J(\vec{w}) = \Phi(\vec{w}) + \gamma \sum_{i=1}^{n_s}\frac{w_i}{w_i+\epsilon}. \]
Then the surrogate function  $g(\vec{w}|\vec{w}^{(m)})$ at the current iterate
$\vec{w}^{(m)}$ must satisfy
\[ 
\begin{aligned} 
   g(\vec{w}|\vec{w}^{(m)})    &\geq  J(\vec{w}) 
      \quad \forall\vec{w} \in \mathcal{W}, \\
g(\vec{w}^{(m)}|\vec{w}^{(m)}) &= J(\vec{w}^{(m)}). 
\end{aligned}
\] 
Granted the existence of this surrogate function, to find the next iterate $\vec{w}^{(m+1)}$ we solve the optimization problem 
\begin{equation}\label{eqn:next1} \vec{w}^{(m+1)} = \argmin_{\vec{w} 
\in \mathcal{W}} g(\vec{w}|\vec{w}^{(m)}). \end{equation}
To show the objective function decreases at the next iterate, observe that the next iterate $\vec{w}^{(m+1)}$ stays within the feasible region and use the two properties of the surrogate function 
\[ J(\vec{w}^{(m+1)}) \leq g(\vec{w}^{(m+1)}|\vec{w}^{(m)}) \leq g(\vec{w}^{(m)}|\vec{w}^{(m)})= J(\vec{w}^{(m)}). \]

To construct this surrogate function, we use the fact that a concave function is below its tangent~\cite[Equation (4.7)]{lange2016mm}. Applying this to our concave penalty $P_\epsilon(\vec{w})$, we have 
\[ 
  P_\epsilon(\vec{w}) \leq P_\epsilon(\vec{w}^{(m)}) + (\vec{w} - \vec{w}^{(m)})^\top \nabla_{\vec{w}} P_\epsilon(\vec{w}^{(m)}), \quad 
\text{for all } \vec{w} \in \mathcal{W}.  \]
With this majorization relation, we define the surrogate function to be 
\[ g(\vec{w}|\vec{w}^{(m)}) = \Phi(\vec{w}) + \gamma\left(P_\epsilon(\vec{w}^{(m)}) +   (\vec{w} - \vec{w}^{(m)})^\top \nabla_{\vec{w}} P_\epsilon(\vec{w}^{(m)})\right)\] 
By dropping the terms that do not depend on $\vec{w}$, it can be readily verified that~\cref{eqn:next1} can be replaced by the equivalent problem 
\begin{equation}\label{eqn:next2}
\begin{aligned}
\vec{w}^{(m+1)} = &\> \argmin_{\vec{w} \in \mathcal{W}} \Phi(\vec{w}) + \gamma \sum_{i=1}^{n_s}\frac{\epsilon w_i}{(w_i^{(m)}+\epsilon)^2} \\
= & \>  \argmin_{\vec{w} \in \mathcal{W}} \Phi(\vec{w}) + \gamma \| \vec{R}(\vec{w}^{(m)})  \vec{w}\|_1  ,
\end{aligned}
\end{equation}
where $\vec{R}(\vec{w}) =
\text{diag}\left(\frac{\epsilon}{(w_1+\epsilon)^2},\dots,\frac{\epsilon}{(w_{n_s}+\epsilon)^2}\right)$.  We see that \cref{eqn:next2} is of the form of a reweighted 
$\ell_1$-optimization problem. The details of the optimization procedure are given
in \cref{alg:penalty}.
(We remark that, in \cref{alg:penalty}, other metrics measuring the difference
between the successive weight vectors can be used in step 3.)

\begin{algorithm}[!ht]
\caption{Reweighted $\ell_1$ Algorithm.} 
\begin{algorithmic}[1]
\REQUIRE Initial guess $\vec{w}^{(0)} \in \mathbb{R}^{n_s}$, stopping tolerance tol, penalty parameters $\gamma,\epsilon \geq  0$. 
\ENSURE  Optimal design $\vec{w}^*\in \mathbb{R}^{n_s}$.
\STATE Initialize $m=1$. 
\STATE Compute $$\vec{w}^{(1)}=\argmin_{\vec{w}\in \mathcal{W}}\Phi(\vec{w})+\gamma\|\vec{w}\|_1.$$ 
\WHILE{$m<m_{max}$ and $\|\vec{w}^{(m)}-\vec{w}^{(m-1)}\|_2>\text{tol}$}
\STATE Update $m=m+1$
   \STATE Compute $$\vec{w}^{(m)}=\argmin_{\vec{w}\in\mathcal{W}}\Phi(\vec{w})+\gamma\sum_{i=1}^{n_s}r_i\cdot w_i,$$ 
   \STATE Update 
       $$r_i=\frac{\epsilon}{(|w_i^{(m-1)}|+\epsilon)^2}, \quad i=1,\ldots,n_s.$$
\ENDWHILE
\STATE Return $\vec{w}^{(m)}=\vec{w}^*.$
\end{algorithmic}
\label{alg:penalty}
\end{algorithm}

We conclude this section with a few remarks regarding this algorithm. In our
application, $\Phi(\vec{w})$ is convex; therefore, each subproblem to update
the design weights is also convex. 
To
initialize the reweighted $\ell_1$ algorithm, 
we start with the weights $r_i = 1$, 
$i = 1, \ldots, \Ns$.
This ensures that, in the first step, we are
computing the solution of the $\ell_1$-penalized optimization problem. 
The subsequent reweighted $\ell_1$ iterations further promote binary designs.  
To solve the subproblems
in each reweighted $\ell_1$ iteration we use an interior point algorithm;
however, any solver for appropriate convex optimization may be used. In
\cref{sec:numerics}, we will provide a discussion of our choice of $\epsilon$
for our application. It is also worth mentioning that besides the penalty
function $P_\epsilon$ used above, another possible choice is
$\sum_{i=1}^{n_s}\arctan(|w_i|/\epsilon)$ which yields the weights
$\frac{1}{|w_i|^2 + \epsilon^2}$.

%
%
\section{Modifed A-optimal criterion}\label{sec:mod_a}

Motivated by reducing the computational cost of computing A-optimal designs, in
this section, we introduce a \emph{modified A-optimality} criterion.  As mentioned in~\cite{ChalonerVerdinelli95}, we can 
consider a weighted A-optimal criterion $\trace{\boldsymbol\Gamma \post (\vec{w})}$, where $\boldsymbol\Gamma$ is a positive semidefinite weighting matrix. Similar to \cref{estobj}, we work with  $\trace{ \boldsymbol{\Gamma}(\post  - \prior{})}$, since the term $\trace{ \boldsymbol{\Gamma}(\prior{})}$ is independent of the weights $\vec{w}$. By choosing $\boldsymbol\Gamma = \prior{-1}$, we obtain the modified A-optimal criterion 
\begin{equation}
\moda{(\vec{w})}\equiv  \trace{\left(\eye+\HH(\vec{w})\right)^{-1} - \eye}.
\label{moda}
\end{equation}
Note that the expression for $\moda{(\vec{w})}$ remains meaningful in the
infinite-dimensional limit. This can be seen by noting that 
\[
\moda{(\vec{w})} = 
\trace{\big(\eye+\HH(\vec{w})\big)^{-1} - \eye} 
= -\trace{\HH(\vec{w})\big(\eye + \HH(\vec{w})\big)^{-1}},
\]
and using the fact that in the infinite-dimensional limit, 
for every $\vec{w} \in [0, 1]^\Ns$,
$\HH(\vec{w})$ is trace class and 
$\big(\eye + \HH(\vec{w})\big)^{-1}$ is a bounded linear operator. 

We show in our numerical results that
the modified A-optimality criterion can be useful in practice, if a 
cheaper alternative to the A-optimal criterion is desired, and minimizing
$\moda{}$ can provide designs that lead to small posterior uncertainty.

\subsection{Derivation of estimators}
Here we seek to improve the efficiency of the modified A-optimal criterion by
computing a randomized estimator for the modified A-optimal criterion and its
gradient.  As in the previous derivation of our estimators, we replace
$\HH(\vec{w})$ by its low rank approximation to obtain the randomized estimator
for the modified A-optimal criterion

\begin{equation}
\moda{(\vec{w})}\approx\trace{(\eye+\hatHH(\vec{w}))^{-1}-\eye} \equiv \modarand{(\vec{w};\ell)}\label{estmoda}.
\end{equation}
Similarly, we use the same low-rank approximation 
${\hatHH}(\vec{w})$ in the gradient of the modified A-optimal criterion to obtain the randomized estimator
\begin{equation}\widehat{\partial_j\Phi}_\text{mod}(\vec{w};\ell) = -\trace{(\eye + {\hatHH}(\vec{w}))^{-1}\FF^T\E_j^\text{noise}\FF(\eye + 
{\hatHH}(\vec{w}))^{-1}},
\label{estmodagrad}
\end{equation}
for $j=1,\dots,n_s$.

\subsection{Computational procedure and cost}
Using similar techniques as those described in \cref{sec:aopt_estimators}, we
can write the estimator for the modified A-optimal criterion in terms of the
eigenvalues of $\T$:

\begin{equation}
\modarand{(\vec{w};\ell)}=-\trace{\D_\T}.
\label{modaexpand}
\end{equation}
Moreover,
\begin{equation}\label{eq:modgradient}
\begin{aligned}
\widehat{\partial_j\Phi}_\text{mod}(\vec{w}) 
&= -\trace{\FF^\top\E_j^\text{noise}\FF} + 2\trace{\FF\V\D_\T\V^\top\FF^\top\E_j^\text{noise}} -\\
&\quad\trace{\FF\V\D_\T^2\V^\top\FF^\top\E_j^\text{noise}}.
\end{aligned}
\end{equation}
with $\D_\T=\eig_\T(\eye+\eig_\T)^{-1}$ and $\V$ defined as in \cref{sec:aopt_estimators}.

The procedure for computing the estimators for the modified A-optimal criterion
follows the steps in \cref{alg:randobjgrad} closely.  Instead of presenting an
additional algorithm, we provide an overview of the computation of the
estimators for the modified A-optimal criterion along with the associated computational cost 
in terms of the number of PDE solves.

To evaluate the estimators for the modified A-optimal criterion, the only precomputation 
we perform is to obtain 
$$s_j=-\trace{\FF^\top\E_j^\text{noise}\FF}, \quad j = 1, \ldots, \Ns.$$
This term appears in the estimator for the gradient and accumulates a total cost of $n_sn_t$ PDE solves.  The remaining terms in the estimators depend on a design $\vec{w}$ and, in particular, the eigenvalues and eigenvectors of ${\hatHH}(\vec{w})$.  As with the estimators for the A-optimal criterion, the eigenvalues and eigenvectors of 
$\hatHH(\vec{w})$ are obtained by \cref{alg:randsubspace}.  Recall that the cost associated with 
\cref{alg:randsubspace}, with $q = 1$, is $4\ell$ PDE solves.  Once the eigenvalues and eigenvectors are computed, \cref{estmoda} can be evaluated without any additional PDE solves.  The remaining computational effort occurs in the evaluation of the gradient.  Because of our modification to the A-optimal criterion, the expression \cref{eq:modgradient} is efficiently evaluated by computing $\FF\V$.  Therefore, the total cost of evaluating the estimators for the modified A-optimal criterion and its gradient is $5\ell$ PDE solves.  From this computational cost analysis, we see the modified A-optimal estimators require $\ell$ less PDE solves than the A-optimal estimators.

\subsection{Error analysis}
We now quantify the absolute error of our estimators with the following theorem. 
\begin{theorem}
Let $\modarand{(\vec{w};\ell)}$ and $\widehat{\nabla\Phi}_\text{mod}(\vec{w};\ell)$ be the randomized estimators approximating the modified A-optimal objective function $\moda{(\vec{w})}$ and its gradient $\nabla\moda{(\vec{w})}$, respectively. Using the notation and assumptions of \cref{objerror}, for fixed $\vec{w}\in[0,1]^{n_s}$
$$\expectation{|\moda{(\vec{w})}-\modarand{(\vec{w};\ell)}|}\leq \trace{f(\eig_2)}  +  \trace{f(\gamma_k^{2q-1}C\eig_2)},$$
and for $j=1,\dots,n_s$,
$$\expectation{|\partial_j \moda{(\vec{w})}-\widehat{\partial_j \Phi}_\text{mod}(\vec{w};\ell)|}\leq 2\|\P_j \|_2\left(\trace{f(\eig_2)} + \trace{f(\gamma_k^{2q-1}C\eig_2)}\right),$$
 where $C$ is defined in \cref{constant}.
\label{modobjerror}
\end{theorem}
\begin{proof} See Appendix \ref{proofmodobjerror}.
\end{proof}

Notice the bounds presented in \cref{objerror} and \ref{modobjerror} differ by
a factor of $\|\Z\|_2$.  Since the modified A-optimal criterion removes one
application of $\prior{}$ from the computation of the A-optimal criterion, the
bounds related to the modified A-optimal criterion no longer have
the factor of $\|\Z\|_2$.

%
%
\section{Numerical results}\label{sec:numerics}

In this section, we present numerical 
results that test various aspects of the 
proposed methods.
We begin by a brief description of the inverse advection-diffusion problem used
to illustrate the proposed OED methods, in \cref{sec:model_problem}.  
The setup of our model problem is adapted from that
in~\cite{a14infinite}, where further details about the forward and
inverse problem can be found. In \cref{sec:model_problem}, we also describle
the numerical methods used for solving the forward problem, as well as the optimization 
solver for the OED problem. 
In \cref{subsec:aoptnumerics}, we test the accuracy of our randomized
estimators and illustrate our error bounds.  Then in
\cref{subsec:rwl1numerics}, we investigate the performance of our proposed
reweighted $\ell_1$-optimization approach.  Next, we utilize the proposed
optimization framework to compute A-optimal designs in
\cref{subsec:computeOED}.  Finally, in \cref{subsec:aoptvsmod}, we compare
A-optimal sensor placements with those computed by minimizing the modified
A-optimal criterion.

\subsection{Model problem and solvers}\label{sec:model_problem} We consider a
two-dimensional time-dependent advection-diffusion equation \[\begin{aligned}
u_t-\kappa\Delta u+\vec{v}\cdot\nabla u = 0 & \hspace{2cm}\text{in
}\domain\times(0,T),\\ u(\cdot,0)=m  & \hspace{2cm}\text{in }\domain,\\
\kappa\nabla u\cdot\vec{n} = 0 & \hspace{2cm}\text{on
}\partial\domain\times(0,T), \end{aligned} \] which models the transport of
contaminants (e.g., in the atmosphere or the subsurface).  Here $\kappa$ is the
diffusion coefficient and is taken to be $\kappa = 0.01$.  The velocity field
$\vec{v}$ is computed by solving a steady-state Navier-Stokes equations, as
in~\cite{a14infinite}.  The domain $\domain$, depicted in \cref{domain}, is the
unit square in $\mathbb{R}^2$ with the gray rectangles, modeling
obstacles/buildings, removed.  The boundary $\partial\domain$ is the union of
the outer boundary and the boundaries of the obstacles.   The PDE is
discretized using linear triangular continuous Galerkin finite elements in
space and implicit Euler in time. We let the final simulation time be $T=5$. 

The inverse problem involves reconstructing the initial state $m$
from space-time point measurements of $u(\vec{x},t)$.
We consider $\Ns = 109$ sensor candidate locations distributed throughout the domain which are indicated
by hollow squares in \cref{domain}.  Measurement data is collected from a subset of these 
locations at three observation times---$t=1,2,$ and $3.5$. 
To simulate noisy observations,
2\% noise is added to the simulated data.    

Recall, in~\cref{sec:background}, we define our prior covariance operator to be a Laplacian-like operator.  Following~\cite{a14infinite}, an application of the square root of the prior on $s\in L^2(\mathcal{D})$ is $v=\mathcal{A}^{-1}s$, which satisfies the following weak form
\begin{equation}\label{eq:prior}
\int_\mathcal{D}\theta\nabla v\cdot\nabla p + \alpha vp\,d\vec{x} = \int_\mathcal{D} sp\,d\vec{x}\qquad\text{for every }p\in H^1(\mathcal{D}).
\end{equation}
Here $\theta$ and $\alpha$ control the variance and correlation length and are chosen to be $\alpha = 0.1$ and $\theta = 0.002$, respectively.

The optimization solver used for the OED problem is a quasi-Newton interior 
point method. Specifically, to solve each subproblem of 
\cref{alg:penalty}, we use \textsc{MATLAB}'s
interior point solver provided by the \verb|fmincon| function; BFGS
approximation to the Hessian is used for line search. We use a vector of all ones, $\vec{1}\in\reals^{n_s}$, as the initial guess for the optimization solver.

\begin{figure}[!ht]
\begin{center}
  \includegraphics[width = 0.5\textwidth]{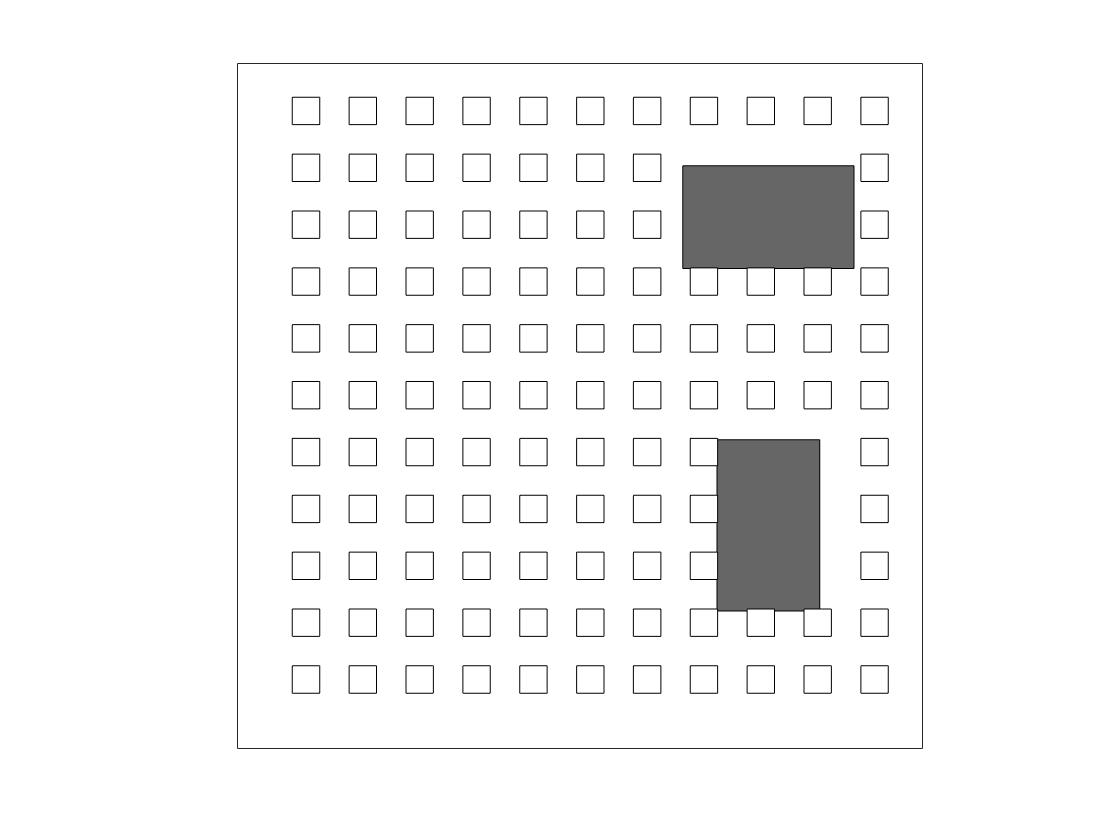}
\end{center}
\caption{Domain with 109 candidate sensor locations.}
\label{domain}
\end{figure}
For the numerical experiments presented in this section, the random matrix
$\boldsymbol{\Omega}$ in Algorithm 1 is fixed during the optimization process.
However, because of the randomness, the accuracy of the estimators, and the
optimal design thus obtained, may vary with different
realizations of $\boldsymbol{\Omega}$. By conducting additional numerical
experiments (not reported here) the stochastic nature of the estimators
resulted in minor variability (at most one or two sensor locations) in the
optimal experimental designs for a modest value of $\ell$. On the other hand,
if $\ell$ is sufficiently large, we observed that the same design was obtained
with different realizations of $\boldsymbol\Omega$

\subsection{Accuracy of estimators}
\label{subsec:aoptnumerics}
Here we examine the  accuracy of the randomized estimators  with respect to 
$\ell$, the
number of columns in the sampling matrix $\randmatrix$ in the randomized
subspace iteration algorithm.  
Specifically, we compute 
\[
\begin{aligned}
e_1(\ell)= \frac{  |\aopt(\vec{w}) - \aoptrand(\vec{w}; \ell)|}{|\aopt(\vec{w})|}, \quad & 
e_2(\ell)= \frac{ \|\nabla\aopt(\vec{w}) - \widehat{\nabla\Phi}_\text{aopt}(\vec{w}; \ell)\|_2}
                {\|\nabla\aopt(\vec{w})\|_2},\\
e_3(\ell)= \frac{ | \moda(\vec{w}) - \modarand(\vec{w}; \ell)|}{|\moda(\vec{w})|}, \quad & 
e_4(\ell)= \frac{ \|\nabla\moda(\vec{w}) - \widehat{\nabla\Phi}_\text{mod}(\vec{w}; \ell) \|_2}
                {\|\nabla\moda(\vec{w})\|_2},
\end{aligned}
\]
with $\vec{w}$ taken to be a vector of all ones; that is, with all sensors
activated.  We let $\ell$ to vary from $17$ to $327$, because the rank of
$\HH(\vec{w})$ is no larger than the number of observations taken,
$n_sn_t=327$.  
\cref{relerr6} illustrates the relative error in the estimators for the
A-optimal criterion and its gradient (left) and the modified A-optimal
criterion and its gradient (right), as $\ell$ is varied.  We  observe that
the  error decreases rapidly with increasing $\ell$.   This illustrates accuracy and
efficiency of our estimators.

Next, we consider the absolute error in the estimators for the objective
function and compare them with the theoretical bounds derived in
\cref{objerror,modobjerror}.  In \cref{bounds}, we compare the absolute error
in the estimators with bounds from \cref{objerror} and \cref{modobjerror}. As
before, we take $\vec{w}=[1,1,\dots,1]^\top\in\reals^{n_s}$.  We observe that
our error bound captures the general trend in the error. Moreover, the error
bound for the modified A-optimality is better since it does not have the
additional factor of $\|\Z\|_2$.

\begin{figure}[!ht]
\begin{center}
\begin{subfigure}{.45\textwidth}
\begin{center}
  \includegraphics[width=2.3in]{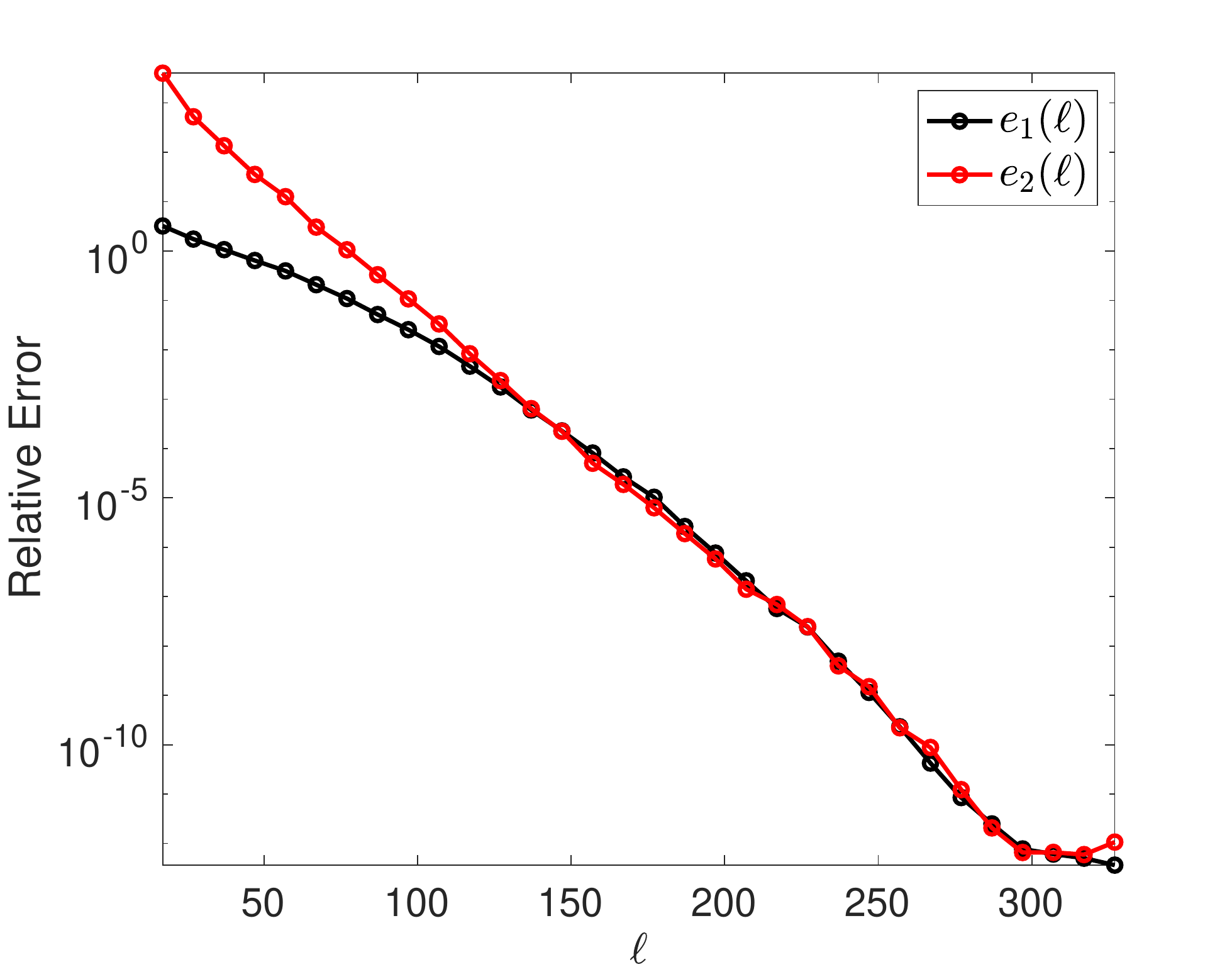}
\end{center}
\end{subfigure}
\begin{subfigure}{.45\textwidth}
\begin{center}
  \includegraphics[width=2.3in]{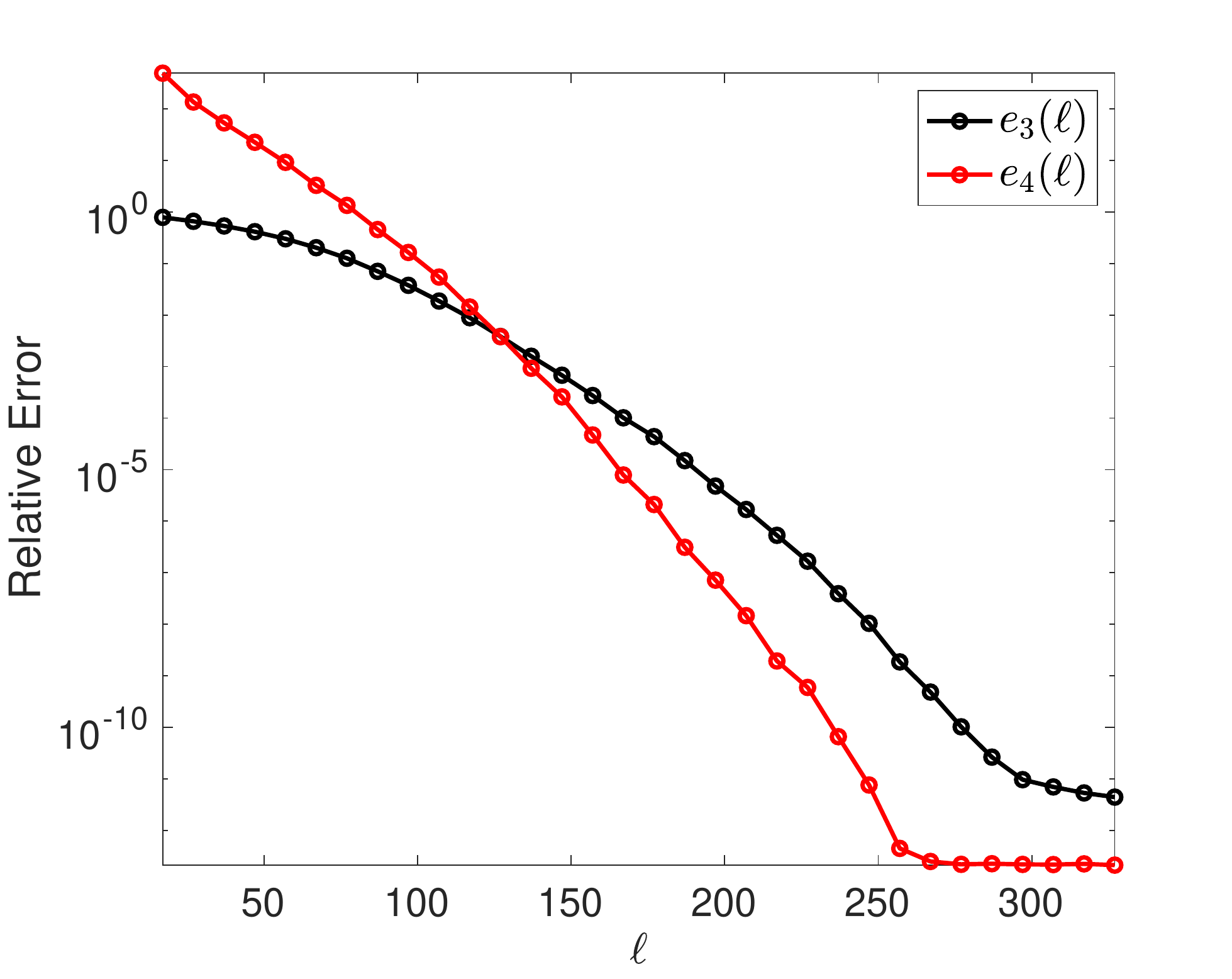}
\end{center}
\end{subfigure}
\caption{Relative error of the randomized estimators for the A-optimal criterion (left), and those corresponding to the modified A-optimal criterion (right) for varying $\ell$.}
\label{relerr6}
\end{center}
\end{figure}

\begin{figure}[!ht]
\begin{center}
\begin{subfigure}{.45\textwidth}
\begin{center}
  \includegraphics[width=2.3in]{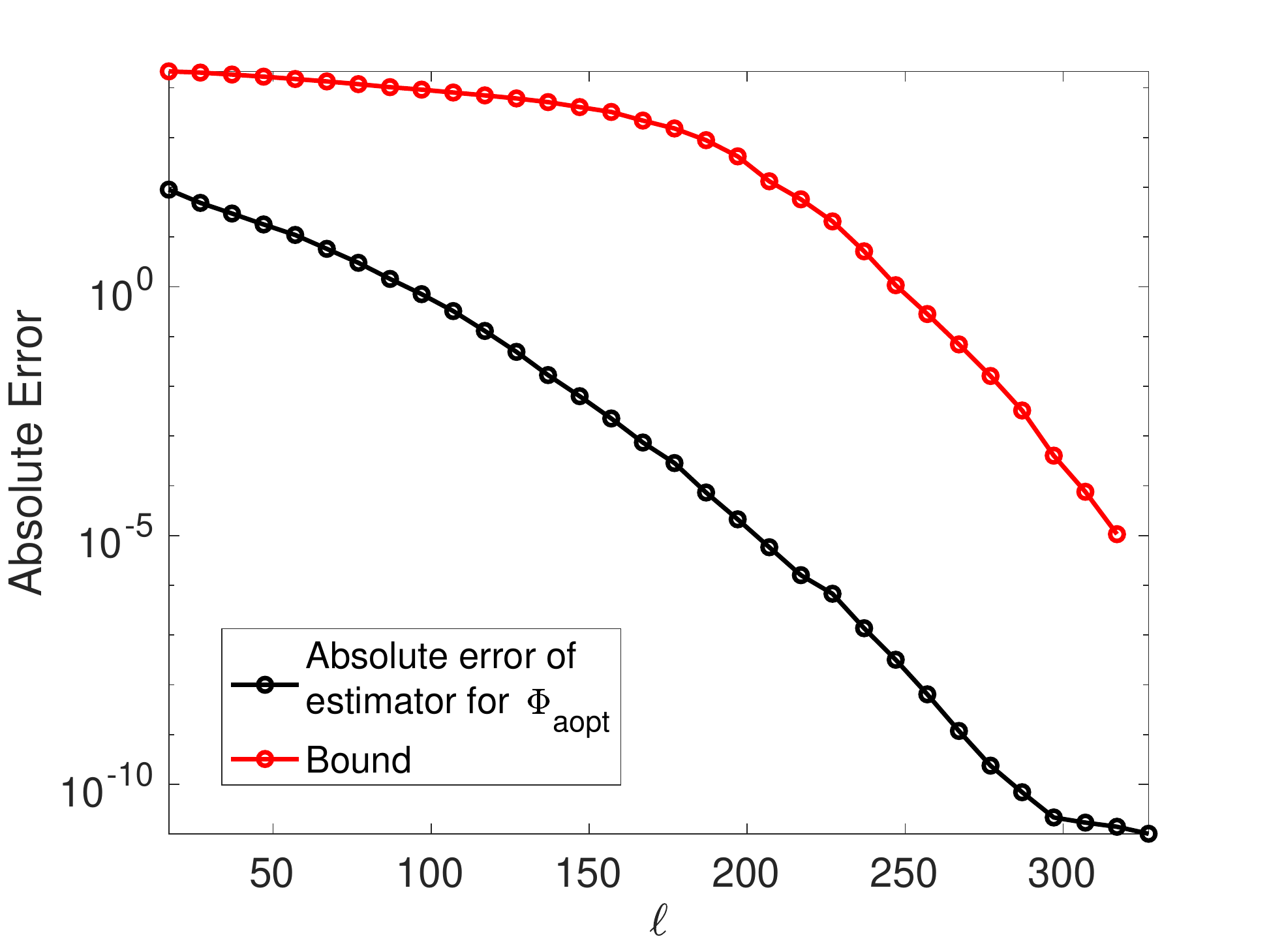}
\end{center}
\end{subfigure}
\begin{subfigure}{.45\textwidth}
\begin{center}
  \includegraphics[width=2.3in]{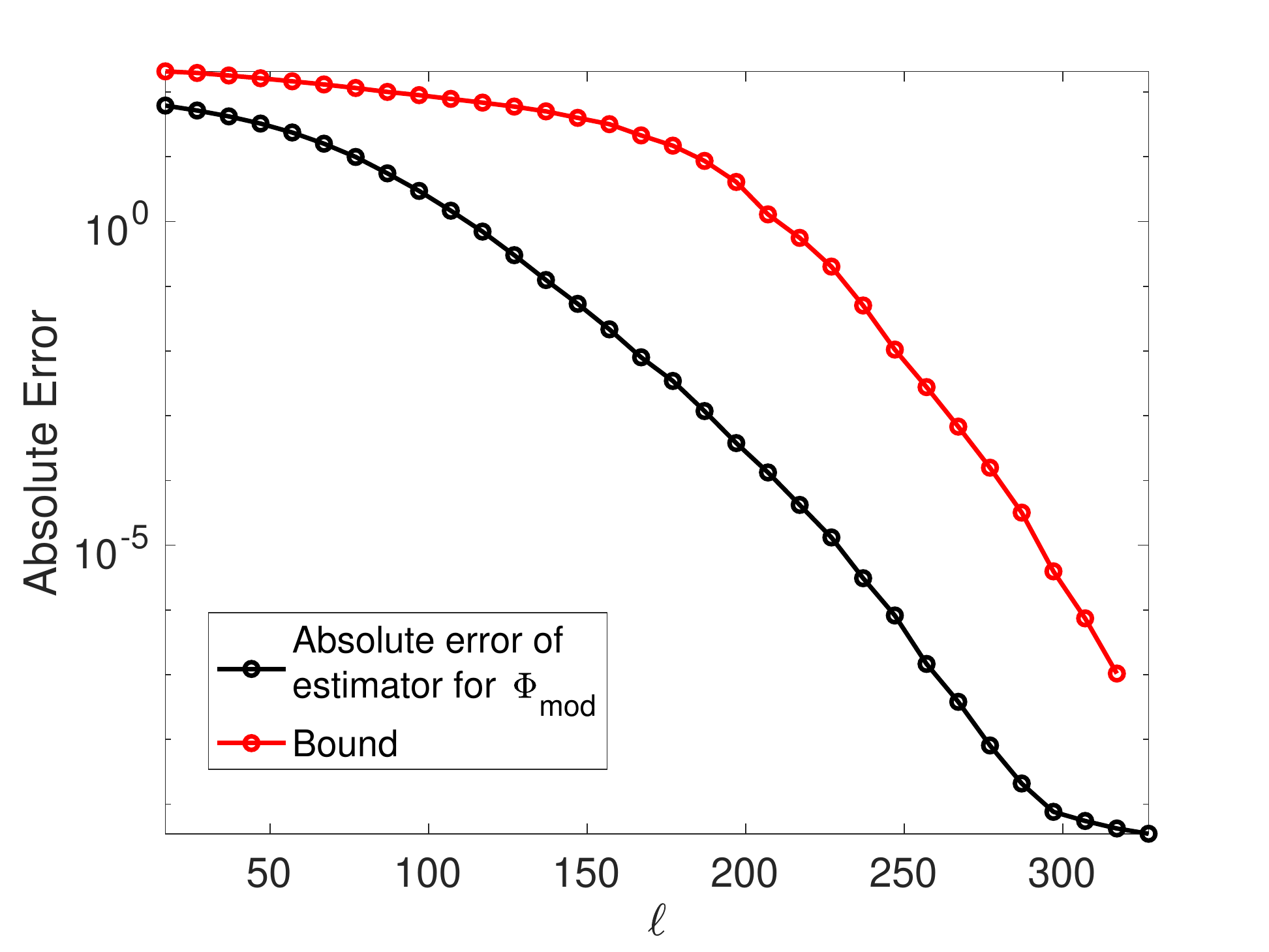}
\end{center}
\end{subfigure}
\caption{Absolute error bound for $\widehat{\Phi}_\text{aopt}(\vec{w})$ from \cref{objerror} (left), and $\widehat{\Phi}_\text{mod}(\vec{w})$ from \cref{modobjerror} (right) for varying $\ell$.}
\label{bounds}
\end{center}
\end{figure}

\subsection{Performance of the reweighted $\ell_1$ algorithm}
\label{subsec:rwl1numerics} We now consider solving \cref{opt} with
\cref{alg:randobjgrad} and \ref{alg:penalty}.   
We first consider the choice of the value of $\epsilon$ in \cref{eqn:next2}.
The user-defined parameter $\epsilon$ controls the steepness of the penalty
function at the origin; see \cref{fig:norm}. However, we observed that if the
penalty function is too steep, the optimization solvers took more iterations
without substantially altering the optimal designs. We found $\epsilon = 1/2^8$ to be
sufficiently small in our numerical experiments, and we keep this fixed for the
remainder of the numerical experiments. 
 
We perform two numerical experiments examining the impact of changing $\ell$
and the penalty parameter $\gamma$ (in \cref{opt}).  In the first experiment,
we fix the penalty coefficient at an experimentally determined value of
$\gamma=3$, and vary $\ell$; the results are recorded in \cref{aopttable}.
Notice when $\ell\geq127$, the objective function value evaluated with the
optimal solution, the number of active sensors, and number of subproblem solves
do not change.  This suggests that the randomized estimators are sufficiently
accurate with $\ell = 127$ and this yields a substantial reduction in
computational cost.

\begin{table}[h!]
\begin{center}
\begin{tabular}{c|c|c|c|c}
$\ell$ & subproblem solves & function count & function value & $ns_\text{active}$ \\
\hline
57 & 10 & 395 & 39.0176 & 38\\
67 & 9 & 610 & 44.7219 & 30\\
77 & 9 & 790 & 44.5484 & 29\\
87 & 9 & 756 & 44.1142 & 30\\
127 & 9 & 1015 & 44.1139 & 30\\
207 & 9 & 978 & 44.1139 & 30\\
307 & 9 & 970 & 44.1139 & 30\\
\end{tabular}
\caption{Number of subproblem solves \cref{eqn:next2}, function evaluations, and active sensors for varying $\ell$ with the reweighted $\ell_1$ algorithm and $n_s=109$.}
\label{aopttable}
\end{center}
\end{table}

The second experiment involves varying $\gamma$, which indirectly
controls the number of sensors in computed designs, with $\ell$ kept fixed.
Here we fix $\ell=207$, which
corresponds to an accuracy on the order of $10^{-7}$ for the A-optimal
criterion (cf. \cref{relerr6}).  
In \cref{rwl1stats}, we report the design weights sorted in descending
order, as $\gamma$ varies. 
This shows that the reweighted $\ell_1$ algorithm indeed produces binary
designs for a range of penalty parameters.  
We also notice in
\cref{rwl1stats2} that as $\gamma$ increases  the sparsity increases (i.e., the
number of active sensors $ns_\text{active}$ decreases) and the number of
function evaluations increases.  The right panel compares the cost of
solving an $\ell_1$-penalized problem for the corresponding penalty parameter  $\gamma$. 
Since this problem is the first iterate of the
reweighted $\ell_1$ algorithm we see that an additional cost is required to
obtain binary designs and this cost increases with increasing $\gamma$ (i.e.,
more sparse designs). 
\begin{figure}[!ht]
\begin{center}
  \includegraphics[width=0.5\textwidth]{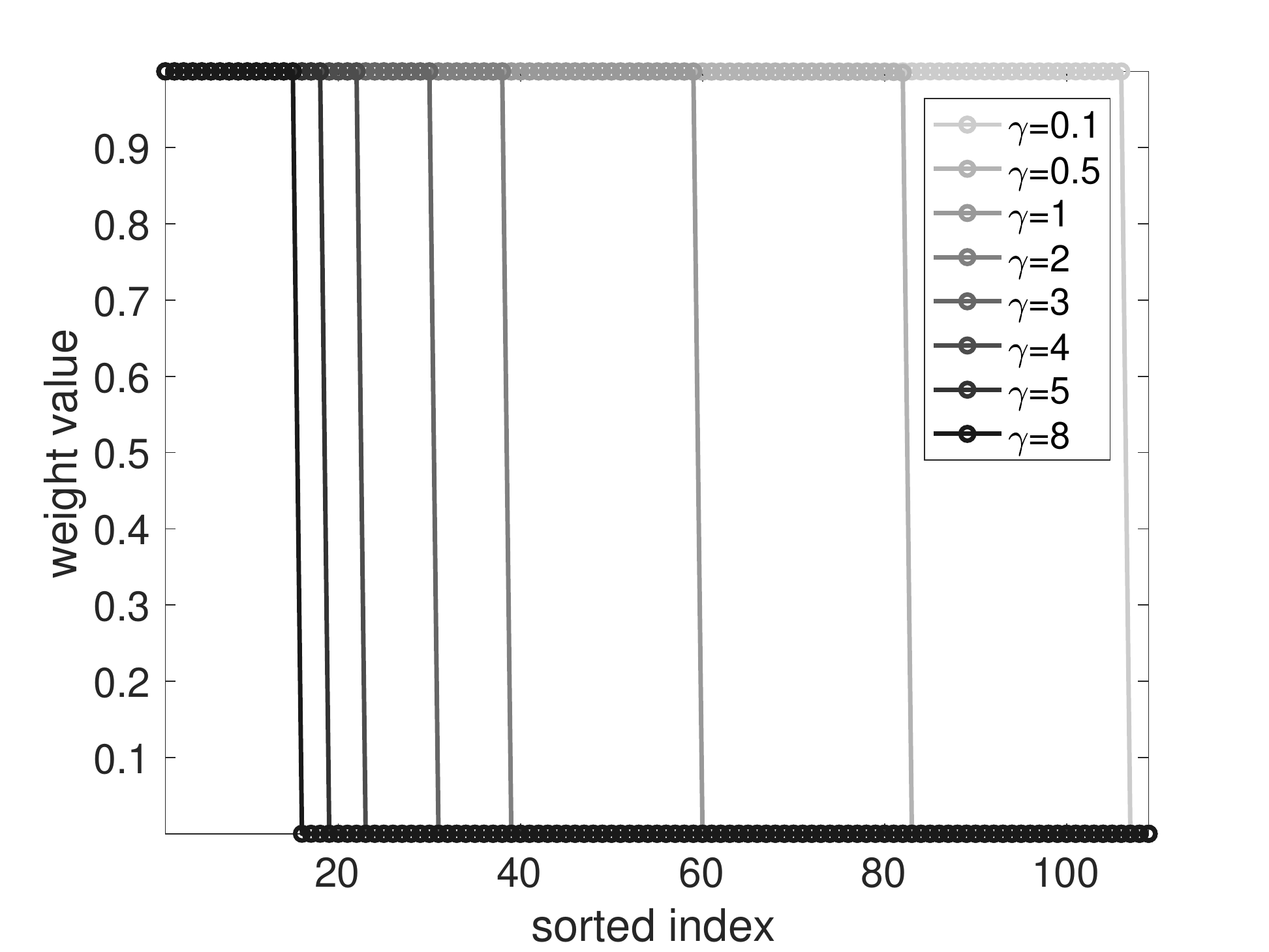}
\end{center}
\caption{Optimal designs as a result of varying $\gamma$ with the reweighted $\ell_1$ algorithm for the A-optimal criterion.  We set $\ell=207$.}
\label{rwl1stats}
\end{figure}
\begin{figure}[!ht]
\begin{center}
\begin{subfigure}{.49\textwidth}
\begin{center}
  \includegraphics[width=1.0\textwidth]{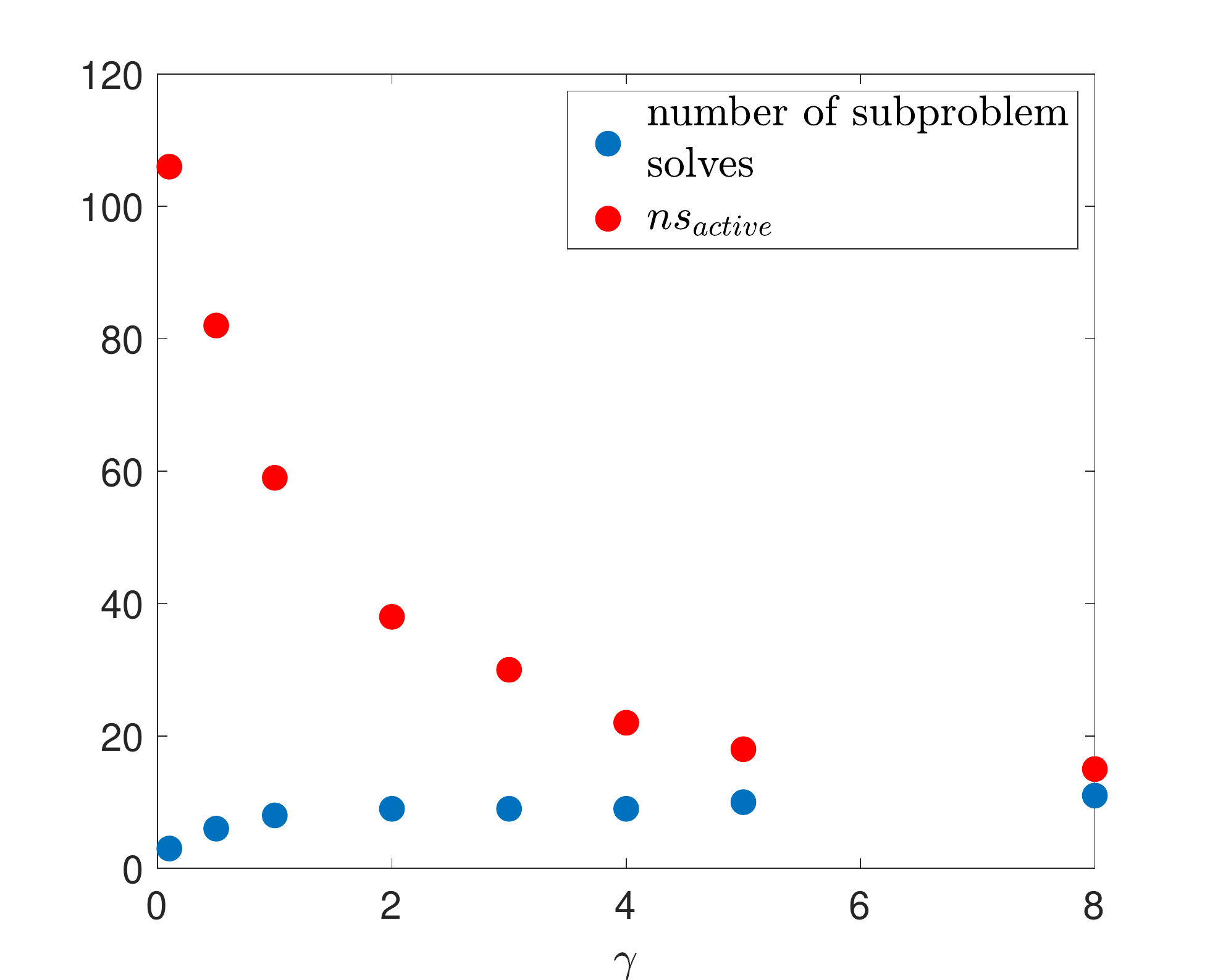}
\end{center}
\end{subfigure}
\begin{subfigure}{.5\textwidth}
\begin{center}
 \includegraphics[width=1.0\textwidth]{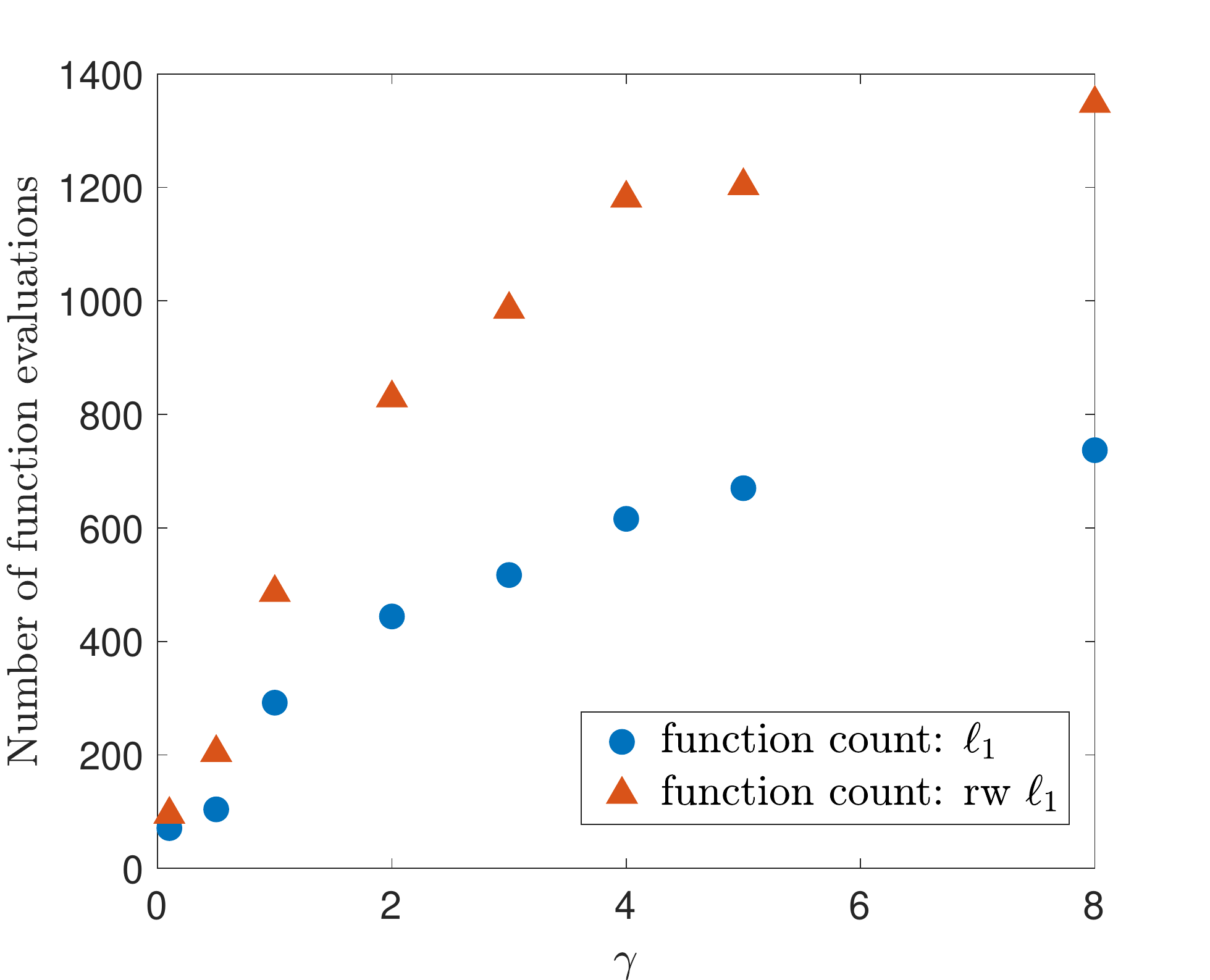}
\end{center}
\end{subfigure}
\caption{The effect of varying $\gamma$ on the reweighted $\ell_1$ algorithm for the A-optimal criterion.  We set $\ell=207$.}
\label{rwl1stats2}
\end{center}
\end{figure}

\subsection{Computing optimal designs}
\label{subsec:computeOED}
In \cref{optdesign}~(left), we report an A-optimal sensor placement obtained 
using our optimization framework, with $\ell=207$ and $\gamma = 5$;
the resulting optimal sensor locations, with $18$ active sensors, 
are superimposed on the posterior standard
deviation field. While the design is computed to yield a minimal average variance of the
posterior distribution, it is also important to consider the mean of this
distribution.  For completeness, in \cref{optdesign} we show the ``true`` initial
condition~(middle panel), used to generate synthetic data, 
and the mean of the resulting posterior distribution for the $18$
active sensor design~(right panel). 
\begin{figure}[!ht]
\begin{center}
\begin{subfigure}{.32\textwidth}
\begin{center}
  \includegraphics[width=1.0\textwidth]{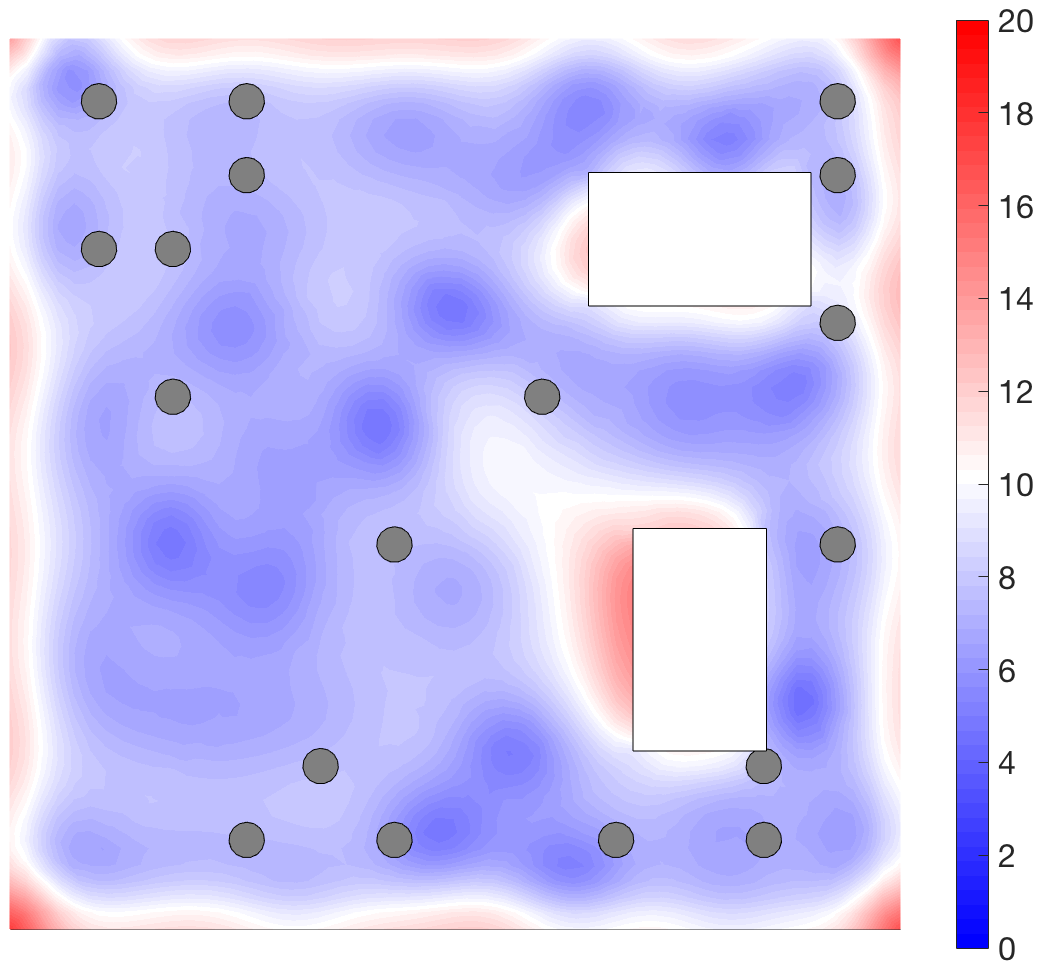}
\end{center}
\end{subfigure}
\begin{subfigure}{.30\textwidth}
\begin{center}
 \includegraphics[width=0.91\textwidth]{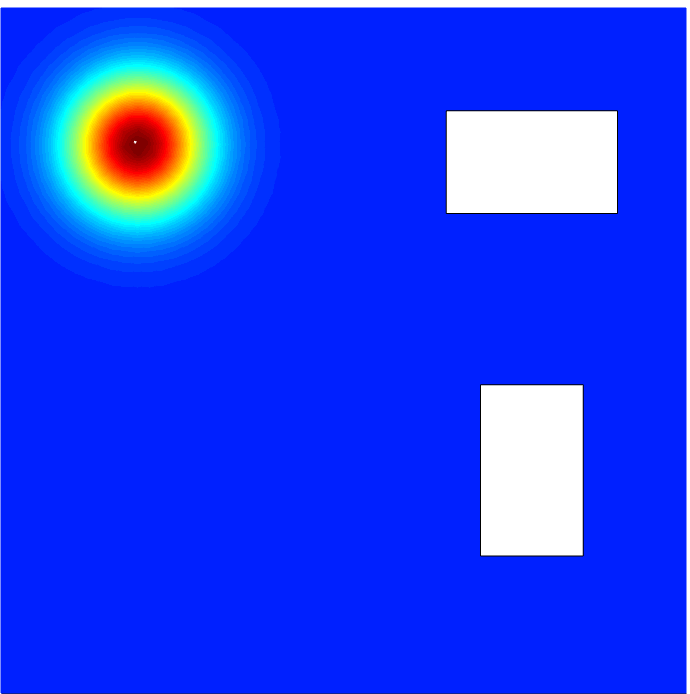}
\end{center}
\end{subfigure}
\begin{subfigure}{.32\textwidth}
\begin{center}
  \includegraphics[width=1.0\textwidth]{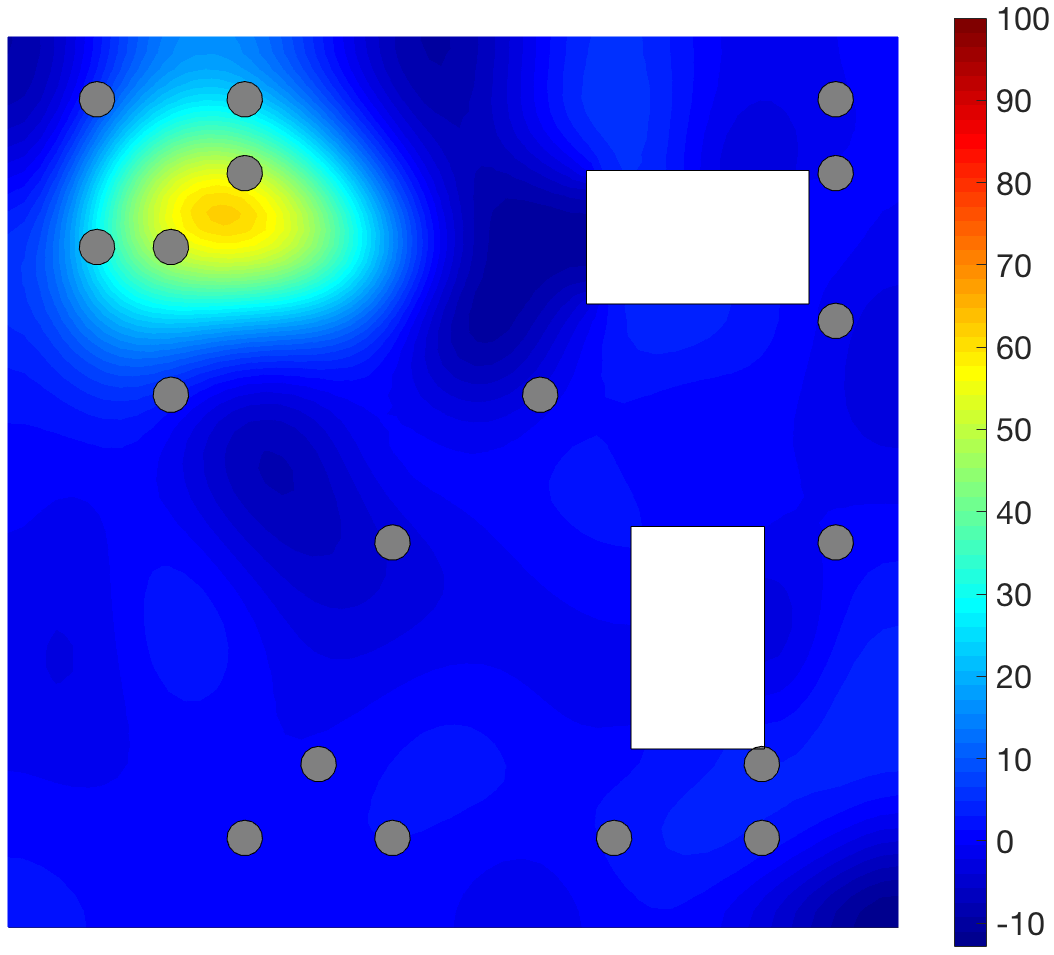}
\end{center}
\end{subfigure}
\caption{Standard deviation computed using the optimal design indicated by the gray circles (left). True initial condition (middle) and initial condition reconstruction (right). The optimal design was computed using $\ell=207$, the reweighted $\ell_1$ algorithm, and $\gamma = 5$.}
\label{optdesign}
\end{center}
\end{figure}
With the $18$ sensor design, we also illustrate the resulting
uncertainty reduction by looking at the prior and posterior standard deviation
fields; see \cref{priorvspost}.

\begin{figure}[!ht]
\begin{center}
\begin{subfigure}{.49\textwidth}
\begin{center}
  \includegraphics[width=1.0\textwidth]{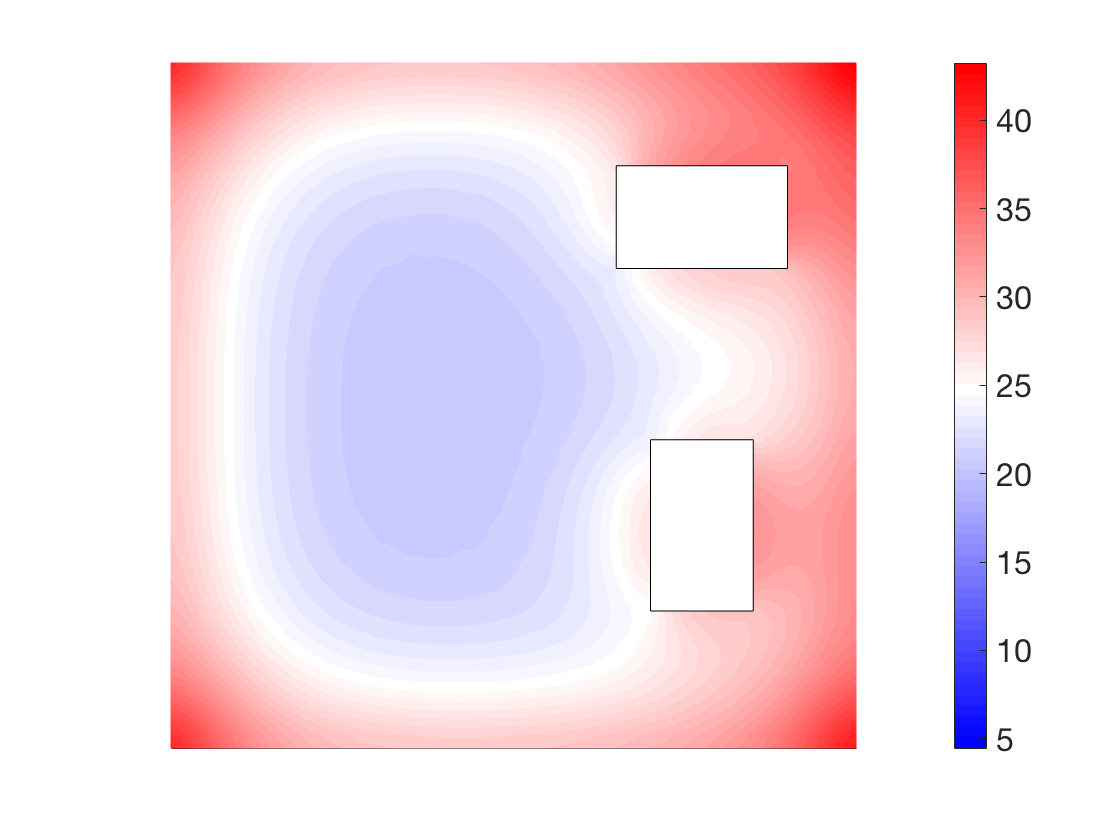}
\end{center}
\end{subfigure}
\begin{subfigure}{.5\textwidth}
\begin{center}
 \includegraphics[width=1.0\textwidth]{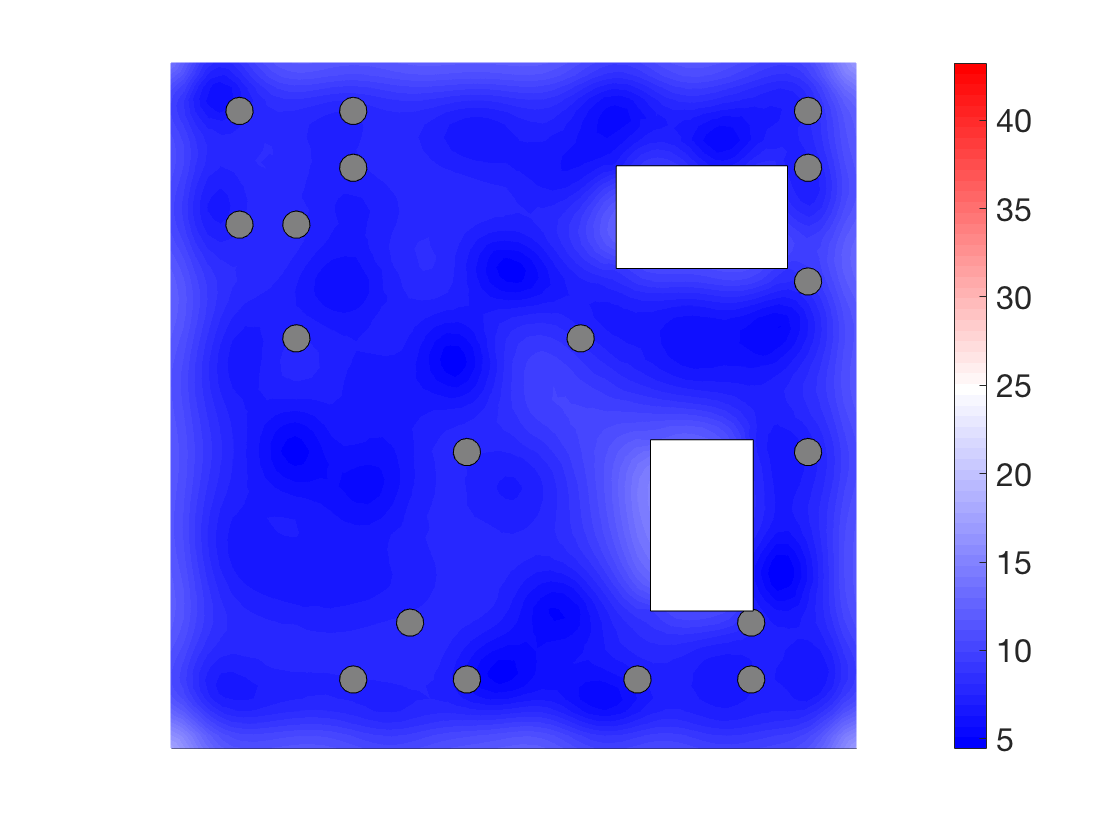}
\end{center}
\end{subfigure}
\caption{Comparison of the prior standard deviation field (left) with the posterior standard 
deviation field (right) computed using the optimal design indicated by the gray circles.}
\label{priorvspost}
\end{center}
\end{figure}

We now compare the designs obtained using the reweighted $\ell_1$ algorithm and our estimators against
designs chosen at random, illustrating the effectiveness of the proposed
A-optimal design strategy. Recall that varying $\gamma$ allows us to obtain
optimal designs with different numbers of active sensors. For each value of
$\gamma$, we use \cref{alg:randobjgrad} and \ref{alg:penalty} to compute an optimal 
design.  We then draw $15$ random designs with
the same number of active sensors as the optimal design obtained using our algorithms.  To enable a consistent comparison, we evaluate the exact A-optimal criterion
$\Phi_\text{aopt}(\vec{w})$ at the computed optimal designs and the random
ones; the results are reported in the left panel of~\cref{compare}. 
The values corresponding to the computed optimal designs 
are indicated as dots on the
black solid line.  The values obtained from the random designs are indicated by
the squares.  We note that the designs computed with the reweighted
$\ell_1$ algorithm consistently beat the random designs, as expected. This 
observation is emphasized in the right panel of~\cref{compare}.  Here, we compared
the computed optimal design (when $\gamma=3$) with 1500 randomly generated designs 
using the exact trace of the posterior covariance.  Again, the computed design results in a lower true A-optimal value than the random designs.

\begin{figure}[!ht]
\begin{center}
\begin{subfigure}{.49\textwidth}
\begin{center}
  \includegraphics[width=1.0\textwidth]{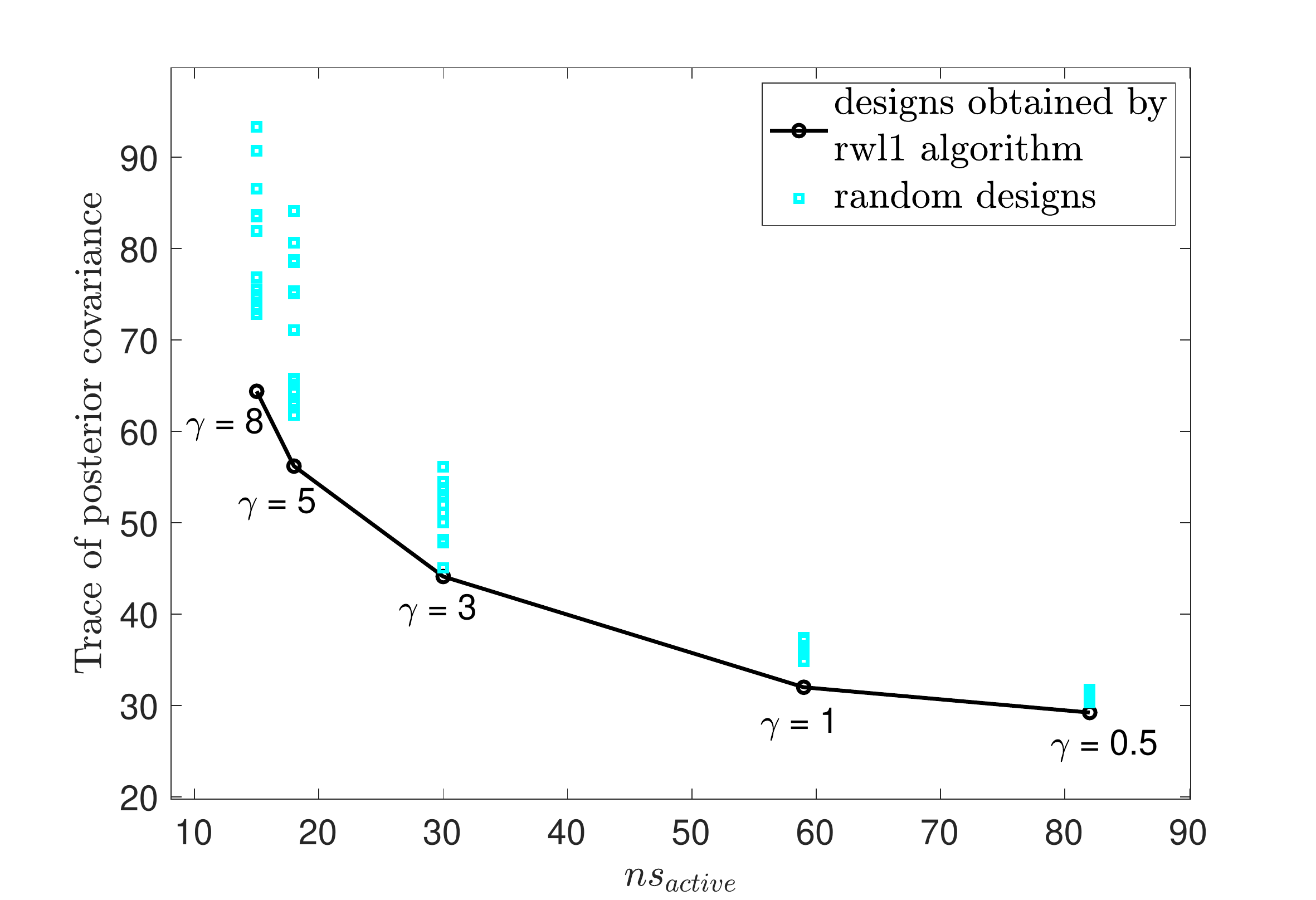}
\end{center}
\end{subfigure}
\begin{subfigure}{.5\textwidth}
\begin{center}
 \includegraphics[width=1.0\textwidth]{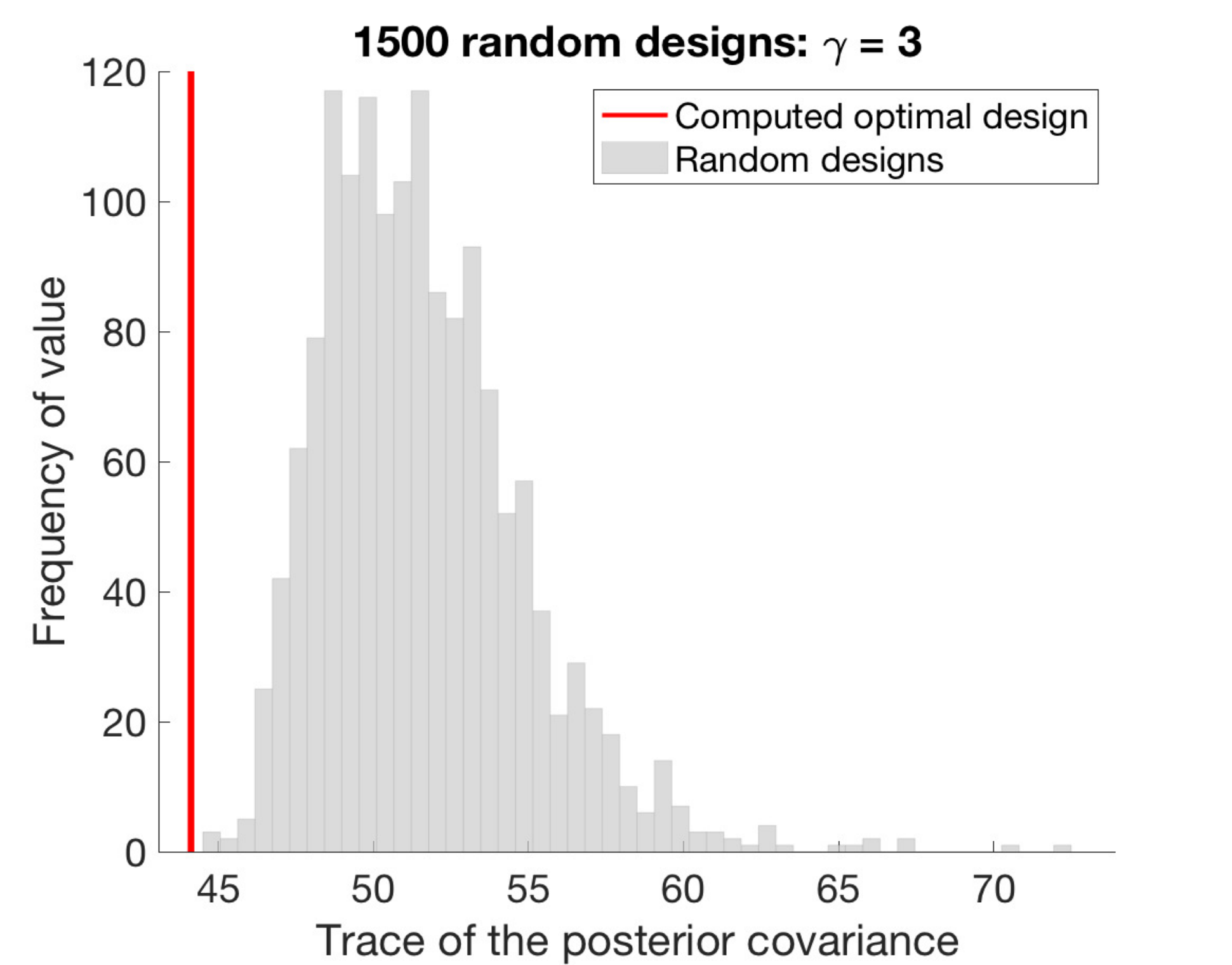}
\end{center}
\end{subfigure}
\caption{The true A-optimal criterion computed using the optimal and randomly generated designs. The optimal designs were computed using the A-optimal criterion and the reweighted $\ell_1$ algorithm for different values of $\gamma$ (left).  Comparing the computed optimal design when $\gamma=3$ to 1500 randomly generated designs (right).}
\label{compare}
\end{center}
\end{figure}

\subsection{Comparing A-optimal and modified A-optimal designs}
\label{subsec:aoptvsmod}
Here we provide a quantitative comparison of sensor placements obtained by
minimizing A-optimal and modified A-optimal criteria using our proposed
algorithms.  Specifically, for various values of $\gamma$, we solve \cref{opt}
with both the A-optimal and modified A-optimal estimators to obtain two sets of
designs.  By varying $\gamma$, the resulting designs obtained with the
A-optimal and modified A-optimal estimators have different number of active
sensors.  Using both sets of designs, we evaluate the exact A-optimal criterion
$\Phi_\text{aopt}(\vec{w})$; these are displayed in \cref{avsmod}.  Observe
that in all cases the computed A-optimal and modified A-optimal designs 
lead to similar levels of average posterior variance.
This suggests that the modified A-optimal criterion could be used as a
surrogate for the A-optimal criterion.  Using the modified A-optimal criterion
decreases the overall number of PDE solves and yields designs that result in
values of the average posterior variance close to those produced by A-optimal
designs.
\begin{figure}[!ht]
\begin{center}
  \includegraphics[width=0.5\textwidth]{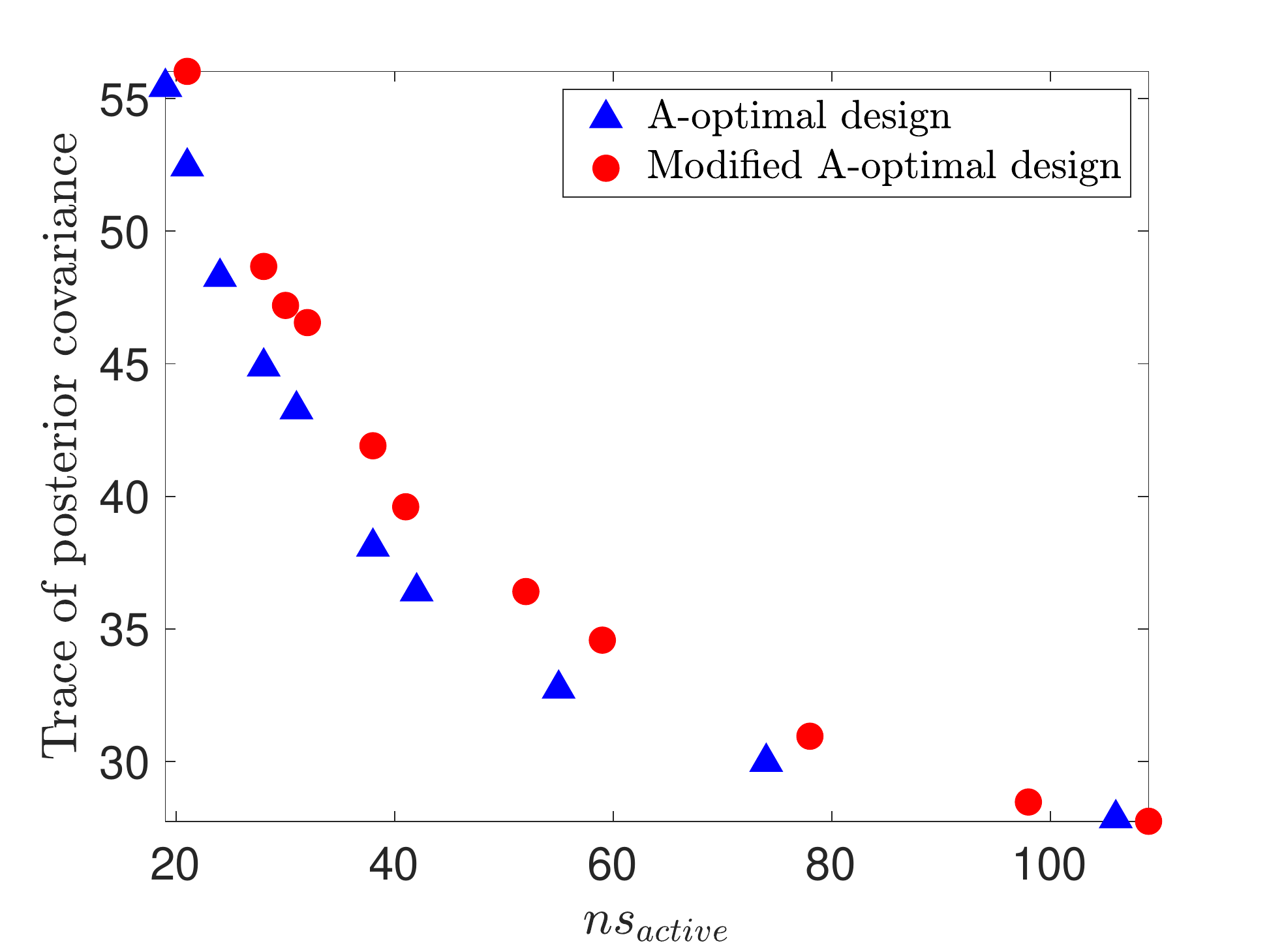}
\end{center}
\caption{Comparison of designs obtained by minimizing the (approximate) 
A-optimal and modified A-optimal criteria.
For each design, we report the exact trace of the corresponding posterior covariance operator.}
\label{avsmod}

\end{figure}

%
%
\section{Conclusion}\label{sec:conclusions}

We have established an efficient and flexible computational framework for
A-optimal design of experiments in large-scale Bayesian linear inverse
problems. The proposed randomized estimators for the OED objective and its
gradient are accurate, efficient, simple to implement and parallelize.
Specifically, the randomized estimators exploit the low-rank structure in the
inverse problem; namely, the low-rank structure of the prior-preconditioned
data misfit Hessian---a common feature of ill-posed inverse problems.  Our
reweighted $\ell_1$-minimization strategy is tailored to sensor placement
problems, where finding binary optimal design vectors is desirable. We also
presented the modified A-optimal criterion, which is more computationally
efficient to compute and can provide designs that, while sub-optimal if the
goal is to compute A-optimal designs, provide a systematic means for obtaining
sensor placements with small posterior uncertainty levels. 

Open questions that we seek to explore in our future work include adaptive
determination of the target rank $k$ within the optimization algorithm, to
further reduce computational costs, while ensuring sufficiently accurate
estimates of the OED objective and gradient. Another possible line of inquiry
is to use different low-rank approximations, such as Nystr\"om's method, and
extending the randomized estimators to approximate trace of matrix functions.
We also seek to incorporate the randomized estimators in a suitable
optimization framework for Bayesian nonlinear inverse problems, in our future
work.

\section*{Acknowledgments}
We would like to thank Eric Chi for useful discussions regarding the MM algorithm.  This material is based upon work supported in part by the National Science Foundation (NSF) award DMS-1745654.

\begin{appendix}
\section{Proofs of bounds}

\subsection{Trace of matrix function}
In the proofs below, we use the Loewner partial ordering~\cite[Chapter 7.7]{HoJ13}; we briefly recapitulate some main results that will be useful in our proof. Let $\A,\B \in \reals^{n\times n}$ be symmetric positive definite; then $\A \preceq \B$ means that $\B - \A$ is positive semidefinite. For any $\S \in \reals^{n\times m}$, it also follows that $\S^\top \A \S \preceq \S^\top \B \S$. Let $\U\eig \U^\top$ be the eigendecomposition of $\A$. Then, $f(\A) = \U f(\eig)\U^\top$ and $\trace{f(\A)} = \sum_{i=1}^n f(\lambda_i)$. If $f$ is monotonically increasing then $\trace{f(\A)} \leq \trace{f(\B)}$ since $\A \preceq \B$ implies $\lambda_i(\A) \leq \lambda_i(\B)$ for $i=1,\dots,n$.

The following bound allows us to bound the trace of a matrix function in terms of its diagonal subblocks. 
\begin{lem}\label{lem:trace} Let $$\A = \bmat{\A_{11} & \A_{12} \\ \A_{12}^\top & \A_{22}}$$ be a symmetric positive definite matrix. Let $f$ be a nonnegative concave function on $[0,\infty)$. Then 
\[ \trace{f(\A)} \leq \trace{f(\A_{11})} + \trace{f(\A_{22})}.\]
\end{lem}
\begin{proof} See Theorem 2.1 and Remark 2.4 in \cite{lee2011extension}.\end{proof}

We are ready to state and prove our main result of this section, 
which is the key in proving \cref{objerror} and \ref{modobjerror}.
 
\begin{thm}\label{thm:main}
Let $\A\in\reals^{n\times n}$ be a symmetric positive definite matrix with eigendecomposition
\[ \A = \U \Lambda \U^\top = \bmat{\U_1 & \U_2} \bmat{\eig_1 \\ & \eig_2} \bmat{\U_1^\top \\ \U_2^\top}, \]
where $\eig_1 = \diag(\lambda_1,\dots,\lambda_k)$ and $\eig_2 = \diag(\lambda_{k+1},\dots,\lambda_n)$ contain the eigenvalues arranged in descending order. Assume that the eigenvalue ratio $\gamma_k \equiv  \frac{\lambda_{k+1}}{\lambda_k} < 1$. 
Let $k$ be the target rank, $p \geq 2$ be the oversampling parameter such that $\ell \equiv k+p \leq n$, and let $q \geq 1$ be the number of subspace iterations.  Furthermore, assume that $\Q \in \mathbb{R}^{n\times \ell} $ and $\T \in  \mathbb{R}^{\ell\times \ell}$ are computed using 
\cref{alg:randsubspace} and define $\Ah \equiv \Q\T\Q^\top$.  Then 
\[ 
\begin{aligned}
0 \leq \expectation{\trace{(\eye+\Ah)^{-1}} - \trace{(\eye + \A)^{-1}}} \leq & \> \trace{f(\eig_2)}  + \trace{f(\gamma_k^{2q-1}C\eig_2)}, 
\end{aligned}\] 
where $f=x/(1+x)$, and the constant $C$ is defined in \cref{constant}.
\end{thm}

\begin{proof}
Suppose $\text{rank}(\A) = r$. Then, 
$\A$ has at most $r$ nonzero eigenvalues, and thus, we can define $\eig_{r-k} = \text{diag}(\lambda_{k+1},\dots,\lambda_r)$ so that 
\begin{equation}\label{eqn:lamrk}  \eig_2 = \bmat{\eig_{r-k} \\ & \mathbf{0}_{n-r-k}}.\end{equation}
 We split  this proof into several steps.
\paragraph{Step 0: Lower bound} 
Let $\tilde\lambda_1\geq \dots \geq \tilde\lambda_{\ell} \geq 0 $ be the
eigenvalues of $\T$ (and also $\Ah$). By the Cauchy interlacing theorem
(see \cite[Lemma 1]{saibaba2016randomized} for the specific version of the
argument), $\lambda_i \geq \tilde\lambda_i $ for $i=1,\dots,\ell$. Using
properties of the trace operator
\[ \begin{aligned}
\trace{(\eye+\Ah)^{-1}} - \trace{(\eye + \A)^{-1}} = & \>    \sum_{i=1}^\ell \frac{1}{1 +\tilde\lambda_i} + (n-\ell) -  \sum_{i=1}^n \frac{1}{1+\lambda_i}\\
= & \> \sum_{i=1}^\ell \frac{\lambda_i - \tilde\lambda_i}{(1+\lambda_i)(1+\tilde\lambda_i)}  + \sum_{i=\ell+1}^n \frac{\lambda_i}{1 +\lambda_i}.
\end{aligned}
\]
Since each term in the summation is nonnegative, the lower bound follows.

\paragraph{Step 1. Trace of matrix function} 
We first write $\Ah = \Q\T\Q^\top = \Q\Q^\top\A\Q\Q^\top = \P_\Q\A\P_\Q$, where $\P_\Q = \Q\Q^\top$ is an orthogonal projection matrix onto the range of $\Q$. Since 
$\Ah$ has the same eigenvalues as $\R \equiv \A^{1/2}\P_\Q\A^{1/2}$~\cite[Theorem 1.3.22]{HoJ13},
\begin{equation}\label{eqn:interchange} \trace{(\eye + \Ah)^{-1}}  = \trace{(\eye + \A^{1/2}\P_\Q\A^{1/2})^{-1}}. 
\end{equation}
Also, since $\P_\Q \preceq \eye$, it follows that $\R =\A^{1/2}\P_\Q\A^{1/2}\preceq \A$ and $\mathbf{0} \preceq \A - \R $. Therefore, from the proof of~\cite[Lemma X.1.4]{bhatia1997matrix} and \cref{eqn:interchange},
\[ \trace{(\eye + \Ah)^{-1}} - \trace{(\eye+\A)^{-1}} \leq \trace{\left( \eye - (\eye+ \A- \R)^{-1} \right)} = \trace{ f(\A - \R)}, \]
where $f(x)$ was defined in the statement of the theorem. 
\paragraph{Step 2. Reducing the dimensionality} Let $\F_\S \equiv \eig_2^q \randmatrix_2\randmatrix_1^\dagger \eig^{-q}$. In \cref{alg:randsubspace}, we compute $\Y = \A^q\randmatrix $ and let $\Y = \Q \R_{\Y}$ be the thin QR factorization of $\Y$. Let $\W_\Q = \R_{\Y}\randmatrix_1^\dagger \eig_1^{-q}(\eye + \F_{\S}^\top\F_{\S})^{-1/2} \in \reals^{\ell \times k}$ be defined as in the  proof of \cite[Theorem 6]{saibaba2016randomized}. It was also shown that $\W_\Q$ has orthonormal columns, so that 
\[ \Q\W_\Q = \U \bmat{\eye \\ \F_\S} (\eye+\F_\S^\top\F_\S)^{-1/2} \in \reals^{n\times k},\]
has orthonormal columns. The following sequence of identities also hold:  $\W_\Q\W_\Q^\top \preceq \eye$, $\Q\W_\Q\W_\Q^\top\Q^\top \preceq \Q\Q^\top$, and
\[ \A - \R \preceq \A - \A^{1/2}\Q\W_\Q\W_\Q^\top\Q^\top \A^{1/2} \equiv \S .\]
Since $f(x)$ is a monotonic increasing function, $\trace{f(\A-\R)}\leq \trace{f(\S)}$. 
\paragraph{Step 3. Split into the diagonal blocks} We can rewrite $\S$ as 
\[ \S = \U \bmat{\S_1 & *\\ * &\S_2} \U^\top ,\]
where $*$ represents blocks that can be ignored and  
\[ 
\S_1 \equiv \eig_1^{1/2}(\eye-(\eye+\F_\S^\top\F_\S)^{-1})\eig_1^{1/2}, 
\qquad 
\S_2 = \eig_2^{1/2}(\eye-\F_\S(\eye+\F^\top_\S\F_\S)^{-1}\F_\S^\top)\eig_2^{1/2}.
\]
We can invoke \cref{lem:trace}, since $f(x) = x/(1+x)$ is concave and nonnegative on $[0,\infty)$. Therefore, we have 
\[\trace{f(\S)} \leq \trace{f(\S_1)} + \trace{f(\S_2)}.\]
Note that the matrix $\U$ disappears, because the trace is 
unitarily invariant.
\paragraph{Step 4. Completing the structural bound} Using an SVD based argument it can be shown that $\eye - (\eye+\F_\S^\top\F_\S )^{-1} \preceq \F_\S^\top\F_\S$, so that 
\[\S_1 = \eig_1^{1/2}(\eye-(\eye+\F_\S^\top\F_\S)^{-1})\eig_1^{1/2} \preceq \eig_1^{1/2}\F_\S^\top\F_\S\eig_1^{1/2}. \]
Therefore, since $f$ is monotonically increasing 
\begin{equation}\label{eqn:s1} 
\trace{f(\S_1)} \leq  \> \trace{ f( \eig_1^{1/2}\F_\S^\top\F_\S\eig_1^{1/2}) } =   \> \sum_{j=1}^k f\left( \lambda_j\left[\eig_1^{1/2}\F_\S^\top\F_\S\eig_1^{1/2}\right] \right) .
\end{equation}
Note that $\F_\S\eig_1^{1/2}$ is $(n-k)\times k$ has at most $\min\{n-k,k\}$ nonzero singular values. Looking more into the structure of $\F_\S\eig_1^{1/2}$, and using \cref{eqn:lamrk}, we can write 
\[ \begin{aligned}
 \F_\S\eig_1^{1/2} = & \> \eig_2^{1/2} \bmat{ \eig_{r-k}^{q-1/2} \\ & \mathbf{0}}  \randmatrix_2\randmatrix_1^\dagger \eig_1^{-q + 1/2} \\
= & \> \eig_2^{1/2} \bmat{ \eig_{r-k}^{q-1/2} \widehat\randmatrix_2\randmatrix_1^\dagger \eig_1^{-q + 1/2} \\ \mathbf{0}},
\end{aligned}\]
where $\widehat\randmatrix_2 \in \mathbb{R}^{(r-k)\times (k+p)}$ such that $ \randmatrix_2 = \bmat{\widehat\randmatrix_2 \\ \mathbf{*}}$. Using the  multiplicative singular value inequalities~\cite[Equation (7.3.14)]{HoJ13} and repeated use of the submultiplicative inequality gives
\begin{equation}\label{eqn:mult}\sigma_j(\F_\S\eig_1^{1/2}) 
\leq \gamma_k^{q-1/2} 
\|\widehat{\randmatrix}_2\randmatrix_1^\dagger\|_2 \sigma_j(\eig_2^{1/2}),
\qquad j=1,\dots,\min\{k,n-k\}.
\end{equation}
The analysis splits into two cases:
\begin{description}
\item [Case 1: $k \leq n-k$.] Since $f$ is monotonically increasing, using \cref{eqn:s1,eqn:mult}
\[ \begin{aligned} \trace{f(\S_1)} \leq & \> \sum_{j=1}^k f\left( \lambda_j\left[\eig_1^{1/2}\F_\S^\top\F_\S\eig_1^{1/2}\right] \right) \leq \sum_{j=1}^k f\left(\gamma_k^{2q-1} \|\widehat\randmatrix_2\randmatrix_1^\dagger\|_2^2 \sigma_j^2(\eig_2^{1/2})\right) \\ 
\leq &\> \sum_{j=1}^{n-k} f\left(\gamma_k^{2q-1} \|\widehat\randmatrix_2\randmatrix_1^\dagger\|_2^2 \sigma_j^2(\eig_2^{1/2})\right) = \trace{f (\gamma_k^{2q-1} \|\widehat\randmatrix_2\randmatrix_1^\dagger\|_2^2 \eig_2)}. 
\end{aligned}
   \]  
\item [Case 2: $k > n-k$.] 
Since $\eig_1^{1/2}\F_\S^\top\F_\S\eig_1^{1/2}$ has at most $n-k$ nonzero eigenvalues, use the fact that $f(0)=0$, along with \cref{eqn:s1,eqn:mult} to obtain
\[ \begin{aligned}\trace{f(\S_1)} \leq & \>  \sum_{j=1}^k f\left( \lambda_j\left[\eig_1^{1/2}\F_\S^\top\F_\S\eig_1^{1/2}\right] \right) = \sum_{j=1}^{n-k} f\left( \lambda_j\left[\eig_1^{1/2}\F_\S^\top\F_\S\eig_1^{1/2}\right] \right)\\
\leq &\> \sum_{j=1}^{n-k} f\left(\gamma_k^{2q-1} \|\widehat\randmatrix_2\randmatrix_1^\dagger\|_2^2 \sigma_j^2(\eig_2^{1/2})\right) = \trace{f (\gamma_k^{2q-1} \|\widehat\randmatrix_2\randmatrix_1^\dagger\|_2^2 \eig_2)}. 
\end{aligned}
\] 
\end{description}
To summarize, in both cases $\trace{f(\S_1)}  \leq \trace{f (\gamma_k^{2q-1} \|\widehat\randmatrix_2\randmatrix_1^\dagger\|_2^2 \eig_2)}.$ Similarly, since $\mathbf{0}\preceq \F_\S(\eye+\F_\S^\top\F_\S)^{-1}\F_\S^\top $, we can show 
\[ \S_2 =  \> \eig_2^{1/2}(\eye-\F_\S(\eye+\F_\S^\top\F_\S)^{-1}\F_\S^\top)\eig_2^{1/2} \preceq \eig_2, \]
so that $\trace{f(\S_2)} \leq \trace{f(\eig_2)}$. 
Combine with step 3 to obtain
\[ \trace{f(\S)} \leq \trace{f(\eig_2) } +\trace{f(\gamma_k^{2q-1} \|\widehat\randmatrix_2\randmatrix_1^\dagger\|_2^2 \eig_2)}.\]
Combine this with the results of steps 1 and 2, to obtain the structural bound 
\[ \trace{(\eye + \Ah)^{-1}} - \trace{(\eye+\A)^{-1}} \leq \trace{f(\gamma_k^{2q-1} \|\widehat\randmatrix_2\randmatrix_1^\dagger\|_2^2 \eig_2)} + \trace{f(\eig_2)} .\]
\paragraph{Step 5. The expectation bound} Note that $\widehat\randmatrix_2 \in\mathbb{R}^{(r-k)\times (k+p)}$ and $\randmatrix_1 \in \mathbb{R}^{k\times (k+p)}$. From the proof of \cite[Theorem 1]{saibaba2016randomized}, we have $\mathbb{E}\,[\|\widehat\randmatrix_2\randmatrix_1^\dagger\|_2^2] \leq C$, where $C$ was defined in~\cref{constant}. By Jensen's inequality, using the fact that $f(x) = x/(1+x)$ is concave on $[0,\infty)$ we have 
\[ \begin{aligned}
\expectation{\trace{(\eye+\Ah)^{-1}} - \trace{(\eye + \A)^{-1}}} \leq & \> \trace{f(\eig_2)} + \expectation{\trace{f(\gamma_k^{2q-1} \|\widehat\randmatrix_2\randmatrix_1^\dagger\|_2^2 \eig_2)}}\\
\leq & \> \trace{f(\eig_2)} + \trace{f(\gamma_k^{2q-1}C\eig_2)}.
\end{aligned}
\]
Combining this with the lower bound (step 0) completes the proof. 
\end{proof}

\subsection{Proof of \cref{objerror}}
\label{proofobjerror}

For the remaining discussion, recall the notation from \cref{objerror} 
 \begin{equation}
\P_j=\FF^\top\E_j^\text{noise} \FF,
\label{notation}
\end{equation} 
where $\FF$ and $\E_j^\text{noise}$ are defined in \cref{eqn:ff} and \cref{Wnoise}, respectively. We will also need 
\begin{lem}[See \cite{alexanderian2018dopt}]
Let $\A,\B\in\reals^{n\times n}$ and let $\B$ be a symmetric positive semidefinite matrix.  Then, we have $|\trace{\A\B}|\leq\|\A\|_2\trace{\B}$.
\label{lemnormtr}
\end{lem}

\begin{proof}[ \cref{objerror}]
Recall our estimator $\aoptrand{(\vec{w};\ell)}$ from \cref{aoptest}.  For fixed $\ell$, using \cref{lemnormtr} we have
$$
\begin{aligned}
\mathbb{E}|\aopt{(\vec{w})}-\aoptrand{(\vec{w};\ell)}|=& \> \mathbb{E}\left|\trace{(\eye+\HH(\vec{w}))^{-1}\MMprior-(\eye+\hatHH(\vec{w}))^{-1}\MMprior} \right|\\
\leq & \> \|\MMprior\|_2 \mathbb{E} |\trace{(\eye+\HH(\vec{w}))^{-1}-(\eye+\hatHH(\vec{w}))^{-1}}|.
\end{aligned}
.$$
Applying \cref{thm:main} establishes \cref{equ:aopt_obj_bound}.

Next, we consider \cref{equ:aopt_grad_bound}. 
Recall the estimator $\gradaoptrand{(\vec{w};\ell)}$ from 
\cref{estgrad}.  We can write the absolute error as
\begin{multline*}
|\partial_j \aopt{(\vec{w})}-\gradaoptrand{(\vec{w};\ell)}|\\
=
|\trace{\big((\eye+\HH(\vec{w}))^{-1}\P_j(\eye +\HH(\vec{w}))^{-1}
 - (\eye+\hatHH(\vec{w}))^{-1}{\P}_j(\eye+\hatHH(\vec{w}))^{-1}\big)\MMprior}|,
\end{multline*} 
where $\hatHH(\vec{w})=\Q\T\Q^\top.$ We use the decomposition
\begin{multline*}
\big((\eye+\HH(\vec{w}))^{-1}\P_j(\eye +\HH(\vec{w}))^{-1} - (\eye+\hatHH(\vec{w}))^{-1}{\P}_j(\eye+\hatHH(\vec{w}))^{-1}\big)\MMprior \\
= -\left(\D \P_j (\eye + \HH(\vec{w}))^{-1}  +  (\eye+\hatHH(\vec{w}))^{-1}\P_j\D\right)\Z, 
\end{multline*}
where $\D \equiv (\eye +\hatHH(\vec{w}))^{-1}  - (\eye+\HH(\vec{w}))^{-1}$. Repeated application of \cref{lemnormtr} gives 
\[\begin{aligned} |\partial_j \aopt{(\vec{w})}-\partial_j \aoptrand{(\vec{w};\ell)}| \leq  & \> \|\P_j\|_2\|\MMprior\|_2 \left(\|(\eye + \HH(\vec{w}))^{-1}\|_2 + \right. \\
& \> \qquad ~~~~~~~~~~~~~~~~~~~~~ \left. \|(\eye + \hatHH(\vec{w}))^{-1}\|_2 \right) \trace{\D}. 
\end{aligned}
\] 

Since $\eye + \HH(\vec{w})$ and $\eye+\hatHH(\vec{w})$ have eigenvalues greater than or 
equal to one, $\|(\eye + \HH(\vec{w}))^{-1}\|_2 + 
\|(\eye+\hatHH(\vec{w}))^{-1}\|_2\leq 2$. Finally, taking the expectation and applying \cref{thm:main}, we have the desired result.
\end{proof}

\subsection{Proof of \cref{modobjerror}}
\label{proofmodobjerror}
\begin{proof}
The proof follows in similar lines as the proof of \cref{objerror}
except the fact that in \cref{modobjerror} we do not 
have $\MMprior$ in the expressions.
\end{proof}

\end{appendix}
\bibliography{references}
\bibliographystyle{siam}

\end{document}